\documentclass[12pt, oneside, reqno]{amsart}
\usepackage{graphicx}
\usepackage{amsfonts }
\usepackage{amsmath}
\usepackage{mathabx}
\usepackage{fullpage}
\usepackage{amssymb}
\usepackage{amsthm}
\usepackage{tikz}
\usepackage{mathtools}
\usepackage{amsfonts,amsmath,mathabx,fullpage,amssymb,amsthm,tikz,mathtools,lmodern,paralist,mathrsfs,hyperref,stmaryrd,enumitem,tikz-cd,ytableau,tcolorbox,cases, soul, accents}
\usepackage[normalem]{ulem}
\usepackage{float}
\usepackage{hyperref}

\usepackage{lmodern}
\usepackage{paralist}
\usepackage[nodisplayskipstretch]{setspace}
\newtheorem{thm}{Theorem}
\newtheorem{conj}{Conjecture}
\newtheorem{lemma}{Lemma}
\newtheorem{proposition}{Proposition}
\newtheorem{defn}{Definition}
\newtheorem{cor}{Corollary}
\newtheorem{ques}{Question}

\theoremstyle{definition}
\newtheorem{rem}[thm]{Remark}
\newtheorem{ex}[thm]{Example}

\newcommand{\B}{\mathcal{B}}

\newcommand{\RR}{\mathcal{R}}
\newcommand{\I}{\mathcal{I}}

\newcommand{\Int}{\text{Int}}

\newcommand{\lex}{<_\text{lex}}

\begin{document}
\title{Dual matroid polytopes and internal activity of independence complexes. }
\author{Alexander Heaton}
\address{Max-Planck Institute for Mathematics in the Sciences and Technische Universit\"at Berlin}
\email{heaton@mis.mpg.de, alexheaton2@gmail.com}
\author{Jos\'e Alejandro Samper}
\address{Max-Planck Institute for Mathematics in the Sciences}
\email{samper@mis.mpg.de, jasamper88@gmail.com}
\date{\today}

\maketitle

\begin{abstract}
Shelling orders are a ubiquitous tool used to understand invariants of cell complexes. Significant effort has been made to develop techniques to decide when a given complex is shellable. However, empirical evidence shows that some shelling orders are better than others. In this article, we explore this phenomenon in the case of matroid independence complexes. Based on a new relation between shellability of dual matroid polytopes and independence complexes, we outline a systematic way to investigate and compare different shellings orders. We explain how our new tools recast and deepen various classical results to the language of geometry, and suggest new heuristics for addressing two old conjectures due to Simon and Stanley. Furthermore, we present freely available software which can be used to experiment with these new geometric ideas.
\end{abstract}

\section{Introduction}

Shelling orders of simplicial/polytopal complexes are a ubiquitous tool used to understand different invariants of interest in combinatorics, algebra, geometry and topology. The main idea is to construct the complex in an organized way by adding one maximal cell at a time. This allows us to easily keep track of how topological and combinatorial data change at each step. 

Many large classes of complexes are shellable and many are not. Since shellable complexes come with many attractive features, a wide variety of techniques are used to detect them. The first prominent example is the proof that boundaries of polytopes are shellable by Bruggesser and Mani \cite{BM1971}. This introduced the now well-established theory of line shellings. Many further results concerning shellability have been discovered and more are likely to come. For an excellent survey see \cite{wachs}. 

Most research along these lines has focused on how to show that a given complex is shellable. However, it is evident a complex may admit many different shelling orders. A natural and seemingly unexplored question concerns the difference between different shelling orders of the same complex. For example, in a simplicial complex a shelling order induces a partition of the faces into boolean intervals. The structure of such partitions may vary according to the shelling order used. While many numerical invariants can be computed with any shelling order, there are pieces of the refined data that change significantly depending on the chosen order. To the best of our knowledge, little has been done to understand these differences. 

In this article we address this issue for matroid independence complexes. Since the first proof of shellability by Billera and Provan \cite{ProvanBillera1980DecompositionsOfSimplicialComplexesRelatedToDiametersOfConvexPolyhedra} it has been known that matroid independence complexes admit many shelling orders. Later, Bj\"orner \cite[Section 7.3]{BjornerHomologyShellabilityofMatroids1992} provided a full characterisation of matroids in terms of orderings on the ground set and induced shelling orders. Furthermore, the decomposition of the face poset of the independence complex is governed by the theory of internal activity, dating back to the work of Tutte for graphs \cite{CrapoTuttePolynomial-1969, Tutte-contribution-chromatic-polynomials-1954}. 

Ordering the ground set allows us to select a special subset of each basis, called the internally passive set. As a set system, the collection of internally passive sets is highly structured, for instance, it is a greedoid \cite{Dawson}. As a partially ordered set, Las Vergnas \cite{LasVergnas} shows it becomes a graded lattice (after attaching an artificial maximum). Bj\"orner realized that this internal activity can be understood as data coming from a shelling order \cite[Section 7.3]{BjornerHomologyShellabilityofMatroids1992}\footnote{The comments to this Section 7.3 reference a 1979 preprint by Bj\"orner that we could not find. On the other hand, it claims the preprint is a draft of sections 7.2-7.5 of the reference given.}. However, varying the order of the ground set produces significant changes in the above-mentioned structures and changing shelling orders yields even more, as we explore in this article. Examples of these differences are elucidated by Dall \cite{DallInternallyPerfect2017}, who defines internally perfect matroids (a class of ordered matroids).

Instead of looking for all shelling orders of an independence complex, which we believe is an extremely hard task, we focus on a more manageable and structured setting. To this end, we explore a new connection with the geometry of matroid polytopes. To connect the two notions we observe that the facets of the dual polytope and the bases of the matroid are in natural correspondence. More precisely, for a matroid $M$ on ground set $E$, the vertices of the matroid polytope $P_M\subset \mathbb{R}^E$ correspond to the set $\B(M)$  of bases (facets) of the independence complex $\I(M)$ and to facets of the dual polytope $P_M^*$ at the same time. This correspondence extends to orders of these sets. For an order $<$ of $\B(M)$ we obtain an order $<_\bullet$ of the facets of $P_M^*$. 

\begin{thm}\label{thdm:shell}
Let $M$ be a matroid and let $<$ be an order of $\B(M)$. If $<_\bullet$ is a shelling order of $P_M^*$, then $<$ is a shelling order of the independence complex $\I(M)$.
\end{thm}

Recall that the braid arrangement $\mathcal{A}_E$ in $(\mathbb{R}^E)^*$ refines the normal fan of the matroid polytope $P_M$, and the maximal chambers of the braid arrangement correspond to total orders of $E$. Using Theorem \ref{thdm:shell} and this connection, we can relate line shellings to internal activity. To see these connections we recall that elements in $(\mathbb{R}^E)^*$ correspond naturally to real-valued functions $\ell$ on the set $E$. Therefore we abuse notation and use $\ell$ freely for either notion. For any subset $A \subset E$ we define the $\ell$-\textit{weight} $\ell(A)$ to be the sum of the values of $\ell$ on elements of $A$. By polyhedral duality, line shellings of $P_M^*$ correspond to orders of the bases $\B(M)$ by using $\ell$-weights.

\begin{thm}\label{thm:lweights-imply-shelling}
Let $M$ be a matroid on the ground set $E$ and let $\ell$ be a generic real-valued function on $E$. The following statements hold:
\begin{enumerate}
    \item[A.] The order $<_\ell$ of $\B(M)$ induced by ordering the $\ell$-weights is a shelling order of $\I(M)$.
    \item[B.] The restriction set of a basis $B$ in the shelling order $<_\ell$ is the internally passive set of $B$ with respect to the total order on $E$ induced by $\ell$.
\end{enumerate}
\end{thm}

We view these results as a geometric analog of Bj\"orner's characterization of matroid independence complexes in terms of shellings.

In Sections \ref{section:extendable shellability-heuristics-for-simons-conjecture} and \ref{section:geometric-approach-to-stanleys-pure-o-sequence-conjecture} we explore the potential of this connection to shed light on two old conjectures in the theory of matroids. The first of these is due to Simon \cite{Simon}, and the second to Stanley \cite{S1977}. Simon's conjecture is about extendable shellability of simplex skeletons and has been the subject of study in at least two recent papers \cite{Dochtermann, Benedetti}. In matroid language, this translates to a statement about the independence complex of uniform matroids. We discuss how Theorems \ref{thdm:shell} and \ref{thm:lweights-imply-shelling} together with Theorem 1.3 in \cite{ArdCasSam} explain the difficulty of this conjecture. We also discuss a related question. Since the matroid polytope of a uniform matroid is a hypersimplex, we posit the following
\begin{conj}\label{conj:hypersimplex-extendably-shellable}
The dual polytope of every hypersimplex is extendably shellable.
\end{conj}
In Section \ref{section:extendable shellability-heuristics-for-simons-conjecture}, the conjectures are shown to be independent, but a resolution on either side has interesting implications for the other. Due to the geometric nature of the new conjecture, we believe it may be more tractable.

Before discussing Stanley's conjecture, we introduce some new terminology. Recall that a (generic) linear functional orders the vertices by their values. Geometrically, we can picture this ordering as a hyperplane sweeping through the matroid polytope. This process can be modified by tilting the normal vector of this hyperplane as it passes through the polytope (granted that the tilts satisfy certain conditions). In this way we obtain many new orderings that shell the dual polytope and hence the matroid independence complex. We call these orders broken-line shellings.

In Section \ref{section:pinned-broken-lines-and-generalized-internal-activity} we make these notions precise and study the restriction sets that arise from these more general shellings. The resulting objects can be thought of as mixing the restriction set posets for several different orders together, with the geometric conditions ensuring that some good combinatorial properties are preserved. Based on these properties, we focus our attention on broken line shellings of a particular type, those whose normal vectors lie in the normal cone of one fixed vertex of $P_M$. We call these \textit{pinned} broken line shellings.

Stanley's conjecture posits that the $h$-vector of a matroid is a pure $\mathcal{O}$-sequence. This means that there is a pure multi-complex whose $F$-vector is exactly the $h$-vector of the matroid. Recall that a multi-complex is a family of monomials closed under divisibility. It has received significant attention and has numerous connections to combinatorial commutative algebra, chip-firing games, matching theory, and more \cite{Merino, Ohcotransversal, zanello, Lattice, DallInternallyPerfect2017, KleeSam2015, DeLoeraKemperKlee}. In Section \ref{section:geometric-approach-to-stanleys-pure-o-sequence-conjecture}, we show how pinned broken line shellings witness an $h$-vector decomposition studied in \cite{KleeSam2015}. This decomposition is the basis for a refined conjecture presented in that article predicting some particular combinatorial data associated to the desired multi-complex.

Typically, to investigate $h$-vectors we use generic initial ideals or a linear system of parameters in a way that destroys much of the combinatorics. It would be beneficial to build up a multi-complex in a way deeply related to the original combinatorics of the matroid. In this direction, we propose that the underlying multi-complex naturally arises from a shelling order.

\begin{conj}\label{conj:friendly-brokenLine} Given a matroid $M$ there exists a pinned broken line shelling order $<$ such that the restriction set poset coincides with the divisibility poset of a pure multi-complex.
\end{conj}

We notice that Dall's proof for internally perfect matroids is a special case of this conjecture. In our new language, he identifies a broad class of matroids for which a line shelling solves Stanley's conjecture.

In Section \ref{section:implementation-SAGE}, we discuss some experimental and freely available software written in \texttt{SAGE}. This software can be used to generate many different pinned broken line shellings for any matroid. In addition, the software automatically constructs the relevant restriction set poset, displaying it graphically along with information about the geometry used to produce it. As explained earlier, we pivot the normal vector of a hyperplane as it sweeps through the matroid polytope, as illustrated in Figure \ref{figure:uniform-matroid-hyperplane-sweep}.
\begin{figure}[!htb]
    \centering
    \includegraphics[width=0.3\textwidth]{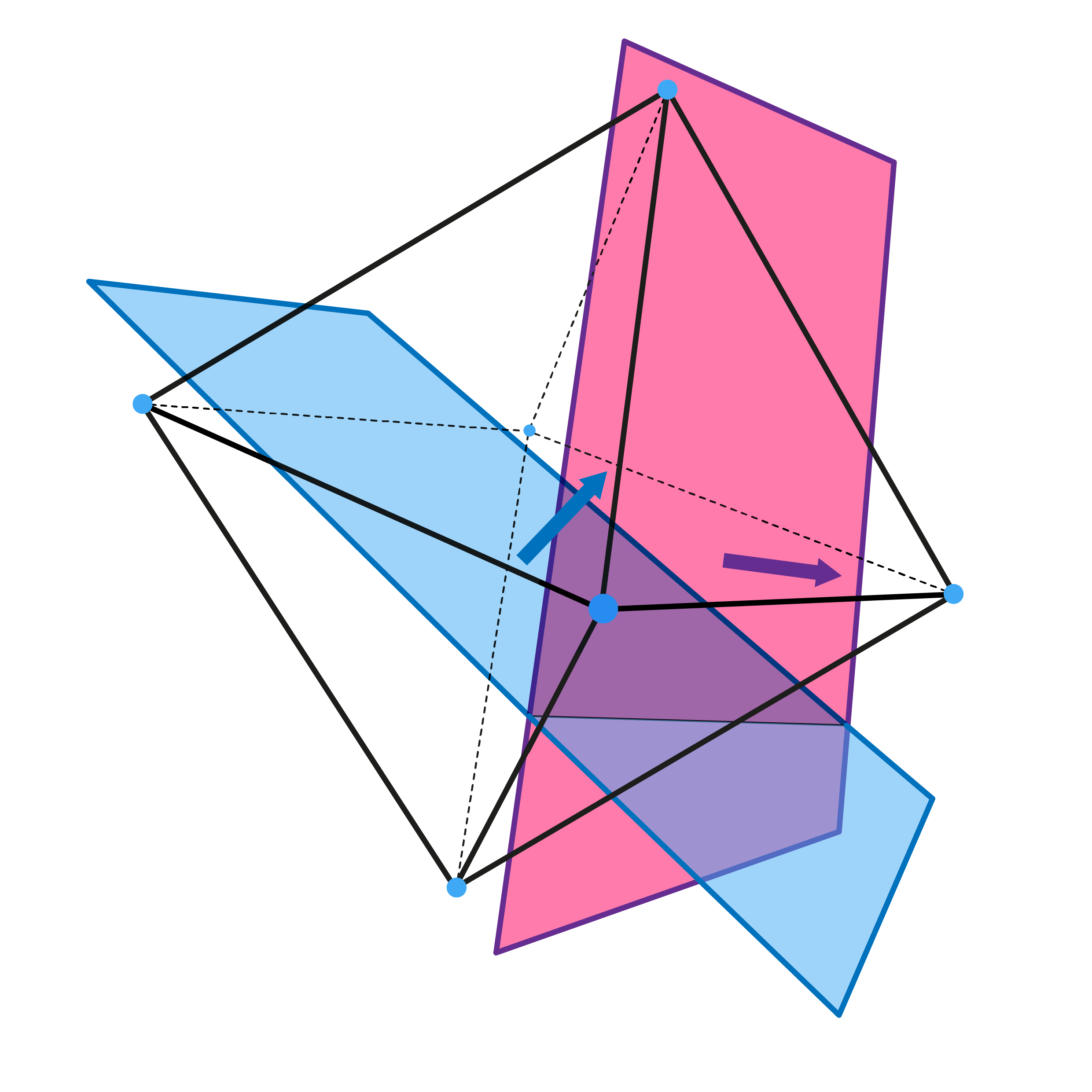}
    \caption{A hyperplane sweeps through the U(4,2) matroid polytope, a hypersimplex}
    \label{figure:uniform-matroid-hyperplane-sweep}
\end{figure}

In Section \ref{section:examples} we see cases where, by using broken line shellings, it is possible to construct a pure multi-complex witnessing Stanely's conjecture, where a simple line shelling does not suffice.

Finally, in Section \ref{section:questions}, we discuss open questions and future directions of research that are suggested by these new methods.

\section{Preliminaries}\label{section:preliminaries}

Most of the terminology is taken from Stanley's \textit{Combinatorics and Commutative Algebra} \cite{StanleyTextCombAndCommAlg1996}, Oxley's \textit{Matroid Theory} \cite{OxleyTextMatroidTheory2011} and Bj\"orner's \textit{The homology and shellability of matroids and geometric lattices} \cite{BjornerHomologyShellabilityofMatroids1992}.

Let $E$ be a finite set. A \emph{simplicial complex} $\Delta$ is a non-empty subset of $2^E$ such that if $F\in \Delta$, then all the subsets of $F$ are also elements of $\Delta$. Elements of $\Delta$ are called \emph{faces}. The \emph{dimension} of a face is one less than its number of elements and the the \emph{dimension} of $\Delta$ is the largest dimension of a face of $\Delta$. A \emph{facet} of $\Delta$ is a face that is not contained in a larger face. We say that a complex $\Delta$ is \emph{pure} if all the facets have the same dimension. 

The $f$-vector of a simplicial complex $\Delta$ is the vector $f(\Delta) = (f_{-1}, f_0, \dots, f_{\dim(\Delta)})$ where $f_i$ is the number of $i$-dimensional faces of $\Delta$. The $h$-vector of $\Delta$ is the integral vector $h(\Delta)= (h_0, h_1, \dots, h_{1+\dim(\Delta)})$ satisfying the equation: 
\[\sum_{j=0}^{\dim(\Delta)+1} h_jx^j = \sum_{j=0}^{\dim(\Delta)+1} f_{j-1}x^j(1-x)^{\dim(\Delta)+1-j}\]

For a complex $\Delta$ let $\mathcal{F}$ denote the set of facets of $\Delta$. If $\mathcal{G}\subseteq \mathcal{F}$, let $\langle \mathcal{G} \rangle$ denote the simplicial complex whose facets are the elements of $\mathcal{G}$. A \emph{shelling order} of a pure simplicial complex $\Delta$ is a total order $F_1<F_2<\dots < F_k$ of $\mathcal{F}$, such that for each $j=2,\dots k$, the intersection of complexes $\langle F_1,\dots, F_{j-1}\rangle\cap \langle F_j\rangle$ is a pure complex of dimension one less than $\dim(\Delta)$. If $F_1<\dots<F_k$ is a shelling order, then for each $j=1,\dots, k$ there is a subset $\mathcal{R}(F_j)$ of $F_j$ such that the faces of $\langle F_1, \dots, F_j\rangle$ not in  $\langle F_1, \dots, F_{j-1}\rangle$ are precisely the subsets of $F_j$ containing $\mathcal{R}(F_j)$. The $h$-vector of the complex becomes a shellability invariant, namely:
\[\sum_{j=0}^{\dim(\Delta)+1} h_jx^j = \sum_{j=1}^k x^{|\mathcal{R}(F_j)|}\]

A matroid is a pair $M=(E, \I)$ where $E$ is a finite set and $\I\subset 2^E$ is a simplicial complex whose faces we dub independent sets. We require that they satisfy the following axiom: if $I,J\in \I$ and $|I|>|J|$, there is $i\in I\backslash J$ such that $J\cup\{i\} \in \I$. The facets of $\I $ are called \emph{bases} and denoted by $\B$. All bases have the same size $r$, the \emph{rank} of the matroid. The complex $\I$ of a matroid admits many shelling orders. If $<$ denotes any total order on $E$, then the induced lexicographic order on $\B$ is a shelling order \textcolor{black}{(each $B \in \B$ corresponds to a square-free monomial)}. This is in fact a defining property for matroids \cite[Section 7.3]{BjornerHomologyShellabilityofMatroids1992}. In this setting, the restriction sets of the shelling order are equivalent to internally passive sets dating back to Tutte \cite{Tutte-contribution-chromatic-polynomials-1954} and Crapo \cite{CrapoTuttePolynomial-1969}. For a total order $<$ on $E$, the internally passive set $IP_<(B)$ of a basis of $M$ is the set of elements $b\in B$ for which one can find $b'<b$ not in $B$ such that $(B\backslash\{b\})\cup\{b'\} \in \B$. \textcolor{black}{In other words, the set $IP_<(B)$ consists of the elements of $B$ that are replaceable with something smaller.} In the lexicographic shelling one has that $\mathcal{R}(B) = IP_<(B)$. We denote by $\Int_<(M)$ the poset whose elements are the bases of $M$ ordered by containment of internally passive sets. 

For a finite set $E$, let $\mathbb{R}^E$ denote the $\mathbb{R}$-vector space of functions from $E$ to $\mathbb{R}$. In coordinates, this vector space corresponds to tuples $(r_e)_{e\in E}$ of real numbers indexed by $E$. For a subset $S \subset B$ let $\chi_S\in \mathbb{R}^E$ be the characteristic function of $S$, namely the $e$-coordinate of $\chi_S$ is equal to one if $e\in S$ and zero otherwise. 

A polytope $P$ in $\mathbb{R}^E$ is the convex hull of finitely many points. For a linear functional $\ell$ on $\mathbb{R}^E$ and a real value $t$ we say that $(\ell,t)$ bounds $P$ if $\ell(x)\le t$ for all $x\in P$. Compactness guarantees $(\ell, t)$ bounds $P$ for large enough $t$. The pair $(\ell, t)$ is said to support $P$ if it bounds $P$ and $(\ell, t')$ does not bound $P$ for any other any $t'<t$. A face of $P$ is a subset $F$ of $P$ such that there is a supporting pair $(\ell,t)$ such that $F=\{x\in P \, | \, \ell(x) = t\}$. Any face of a polytope is again a polytope. The dimension of a polytope is the dimension of the smallest affine space containing $P$. A facet of $P$ is a face of dimension one less than $P$. A polytope has finitely many faces that fit into a poset, where the order is by inclusion. 

The normal cone $N(F,P)$ of a face $F$ of a polytope $P \subset \mathbb{R}^E$ is the closure in $(\mathbb{R}^E)^*$ of the set of linear functionals $\ell$ for which the unique supporting pair $(\ell, t)$ defines $F$. For faces $G,F$ of $P$ the containment $G\subset F$ is equivalent to $N(F,P) \subset N(G,P)$. Thus the normal cones ordered by inclusion define the opposite of the face poset of $P$. This opposite face poset turns out to be the face poset of another polytope which we call the dual polytope. In this article, we need the existence of the dual polytope together with its combinatorial structure, so we abuse notation and ignore the coordinates. 

A shelling order of a polygon is an ordering of the edges such that each edge that is not the first shares a vertex with a previous edge. A shelling of a polytope $P$ of dimension at least $3$ is an order $F_1,..., F_k$ of the facets, such that for $j\ge 2$, $F_j\cap \left(\bigcup_{i=1}^{j-1}F_i\right)$ is a union of facets of $F_j$ that is an initial segment of a shelling order (see \cite[Ch. 8]{ZieglerP} for more details). It is well-known that shelling orders are combinatorial invariants: two polytopes with isomorphic face lattices have the same shelling orders. Bruggesser and Mani \cite{BM1971} showed that every polytope is shellable by using line shellings. For convenience, we introduce them here in their dual setting. The facets of the dual polytope $P^*$ correspond to vertices of $P$. A line shelling of $P^*$ is obtained by ordering the vertices according to their value under a generic linear functional.

Let $M=(E,\I)$ be a matroid and let $P_M=\text{conv}\{\chi_B \, | \, B\in \B\}\subset\mathbb{R}^E$. We call $P_M$ the matroid polytope of $M$. Matroid polytopes were characterized by Gelfand, Goresky, MacPherson and Serganova \cite{GelfandGoreskyMacPhersonSerganova1987CombinatorialGeometriesConvexPolyhedraSchubertCells} in terms of their edges: they are the polytopes whose vertices are characteristic vectors and whose edges are parallel to vectors of the form $\chi_e-\chi_{e'}$ for some $e, e'\in E$. The set of edges (one dimensional faces) of the polytope $P_M$ defines a graph $G_\B$ whose vertices are bases where an edge is formed if they differ by one element. The graph $G_\B$ is often referred to as the dual graph of the independence complex.

Matroid polytopes can be equivalently characterized in terms of the normal cones of their faces. To explain this we need ordered partitions. An ordered partition $\pi :=(E_1\prec\dots \prec E_k)$ of $E$ is a collection of disjoint sets whose union is $E$. The relatively open cell of the braid cone in $(\mathbb{R}^E)^*$ associated to the ordered partition $\pi:=(E_1\prec\dots \prec E_k)$ is the set of all linear functionals $\ell$ such that $\ell(e)=\ell(e')$ if $e, e' \in E_i$ and $\ell(e) < \ell(e')$ whenever $e\in E_i$, $e'\in E_j$ and $E_i\prec E_j$. In particular, full dimensional cells of the braid cone are in one to one correspondence with linear orders on $E$. The braid arrangement $\mathcal{A}_E$ is the collection of the closures of these cells.

A polytope $P\subset \mathbb{R}^E$ is a matroid polytope if and only if all its vertices are characteristic vectors and every normal cone is a union of cones of the braid arrangement $\mathcal{A}_E$. In particular, the braid cone associated to an ordering of $E$ is entirely contained in the normal cone of a vertex.

\section{Shelling orders of dual matroid polytopes and independence complexes}\label{section:shelling-orders-dual-matroid-polytopes-and-independence-complexes}

In this section we prove Theorem \ref{thdm:shell} which establishes a relationship between shelling orders of the dual matroid polytope $P_M^*$ and the independence complex $\I(M)$. We also explore a few of its basic consequences. 

For a basis $B$ let $F_B$ denote the facet of $P_M^*$ associated to $B$ under the natural correspondence explained above. Recall for a linear order $<$ on $\B$ let $<_\bullet$ be the corresponding order of the facets of $P_M^*$.

\begin{proof}[Proof of Theorem \ref{thdm:shell}]
Assume that $<$ is an order on $\B(M)$ such that $<_\bullet$ is a shelling order of $P_M^*$. Let $B'<B$ be two bases. We need to show that $B\cap B'$ is contained in a facet of the intersection of $\langle B \rangle \cap \langle D, D<B\rangle$ whose size is $|B| -1$. Since $<_\bullet$ is a shelling order of $P_M^*$, there is a basis $B''<B$ such that $F_{B''}\cap F_B$ is a codimension one face of $F_B$, that is a facet of $\left(\bigcup_{D< B}F_D\right)\cap F_B$ containing $F_{B'}\cap F_{B}$. Consider the linear functional $\ell\in(\mathbb{R}^E)^*$ such that $\ell(e) =1$ if $e \in B \cap B'$ and $\ell(e) = 0$ otherwise. This functional $\ell$ is maximized by the vertices $\chi_B$ and $\chi_{B'}$ of $P_M$, and therefore defines a face $G$ of $P_M$ containing both vertices. The dual of this face $\hat G$ in $P_M^*$ is therefore contained in the intersection of facets $F_{B'}\cap F_B \subseteq F_{B''}\cap F_B \subseteq F_{B''}$. Therefore $\chi_{B''}$ is also vertex of $G$. Now notice that $|B\cap B'|= \ell(B) = \ell(B'') = |B''\cap(B\cap B')|$. Hence $B' \cap B \subseteq B''\cap B$. Since $F_B$ and $F_{B''}$ are connected in codimension $1$, $\chi_B$ and $\chi_{B''}$ are connected by an edge of $P_M$ and therefore $|B\cap B''|= |B|-1$ as desired. 
\end{proof}

This theorem recovers some classical theorems in the theory of matroids.
\begin{enumerate}
    \item Matroid independence complexes are shellable \cite{ProvanBillera1980DecompositionsOfSimplicialComplexesRelatedToDiametersOfConvexPolyhedra}.
    \item The lexicographic ordering of the bases always produces a shelling order \cite{BjornerHomologyShellabilityofMatroids1992} (see Section \ref{section:line-shellings-and-internal-activity}).
    \item Shelling orders of matroids are sometimes reversible \cite{Chari1996ReversibleShellingsAndInequalityForHVectors} (see Proposition \ref{prop:nonRev}).
\end{enumerate}
Our interest in this theorem is that it provides a large class of shelling orders of $\I(M)$ coming from geometry, which we may understand better.

\begin{defn}
A shelling order $<$ of the independence complex of a matroid $M$ is called geometric if $<_\bullet$ is a shelling order of $P_M^*$.
\end{defn}

We can try to obtain similar results for other types of simplicial complexes. However, the polytopes would need to be quite special. Given a simplicial complex or a polytope (or a pure polytopal complex) $\Gamma$ we define two graphs $G_{\Gamma}$ and $H_\Gamma$ in the vertex set $\mathcal{F}$, the facets of $\Gamma$. In $G_\Gamma$ two facets are connected if they intersect in codimension $1$ and in $H_\Gamma$ two facets are connected if they intersect in a nonempty face. These are called the dual graph and the intersection graph, respectively. A similar argument to that of Theorem \ref{thdm:shell} can be used to show the following, where the graphs replace the duality arguments.

\begin{thm}\label{thm:veryGeneral}
Let $\Delta$ be a pure simplicial complex and let $\Gamma$ be a pure polytopal complex (or polytope). Assume that there is a bijection $\varphi: \mathcal{F}(\Delta) \to \mathcal{F}(\Gamma)$ that induces a bijection $G_\Delta \to G_\Gamma$ and an embedding $H_\Delta \to H_\Gamma$. Then the preimage under $\varphi$ of any shelling order of $\Gamma$ is a shelling of $\Delta$.
\end{thm}

Suppose  that $\Delta$ is any complex and $\Gamma$ is the boundary of the dual polytope to the convex hull of characteristic vectors of the facets of $\Delta$: then Theorem~\ref{thm:veryGeneral} applies if and only if $\Delta$ is the independence complex of a matroid. The condition on the graph is equivalent to the matroid cryptomorphism in terms of polytopes and edge directions. This however, does not rule out finding similar techniques for other complexes by using a suitable polytope or polytopal complex. 

One may ask if the converse is true, that is whether any shelling order of $\B(M)$ is a shelling order of $P_M^*$. Intuitively this should be false: the intersection conditions for shelling orders at the level of simplicial complexes are far less restrictive than their polyhedral counterparts. To confirm our intuition, recall that a shelling order is called \emph{reversible} if the opposite order is also a shelling order. 

\begin{proposition}\label{prop:nonRev}
There exists a matroid $M$ and a shelling order $<$ of $\I(M)$ such that $<_\bullet$ is not a shelling order of $P_M^*$.
\end{proposition}

\begin{proof}
Lemma 8.10 of \cite{ZieglerP} says that every shelling order $<_\bullet$ of $P_M^*$ is reversible. Thus it suffices to find a shelling order of $\I(M)$ that is not reversible for some matroid $M$. There are several ways to achieve this, here we use just one. Let $M$ be any rank $d$ matroid that contains three pairwise disjoint bases $B$, $B'$ and $B''$. Consider any ordering of $E$ such that $B$ is the lexicographically smallest facet. By Lemma 3.2 of \cite{KleeSam2015}, both $B'$ and $B''$ are homology facets in the lexicographic shelling $<$ associated to the order. Hence there is a shelling order $\hat<$ of $\I(M)$ that coincides with $<$ on every basis different from $B', B''$ and such that the last two bases in the order are $B'$ and $B''$. The reversal of this order is not a shelling order since the first two bases on the list are disjoint. Therefore $\hat <$ is not geometric.
\end{proof}

\begin{defn}\label{definition:weakly-geometric}
A shelling order $<$ of $\I(M)$ is called weakly geometric if there is a geometric shelling $\hat <$ of $\I(M)$ such that the $<$-restriction set of any basis $B$ equals the $\hat{<}$-restriction set of $B$.
\end{defn}

\begin{rem}
Notice that the shelling orders obtained in the proof of Proposition \ref{prop:nonRev} have the same restriction sets as the original lexicographic shelling. We will see in the next section that lexicographic shelling orders are always geometric, thus our examples of non-geometric shelling orders are weakly geometric.
\end{rem}

\begin{ex}\label{ex:independenceButNotPolytope}
The uniform matroid $U_{4,2}$ is the matroid whose set of bases is $\binom{[4]}{2}$. Its matroid polytope is combinatorially equivalent to an octahedron and thus its dual matroid polytope is combinatorially equivalent to a cube. The edge ordering $12 < 23 < 34 < 14 < 13 < 24$ is a shelling order of $\I(U_{4,2})$ that is not a shelling order of the cube. Furthermore, the restriction sets and the corresponding poset for this shelling order are given by
\begin{center}
    \includegraphics[width=0.6\textwidth]{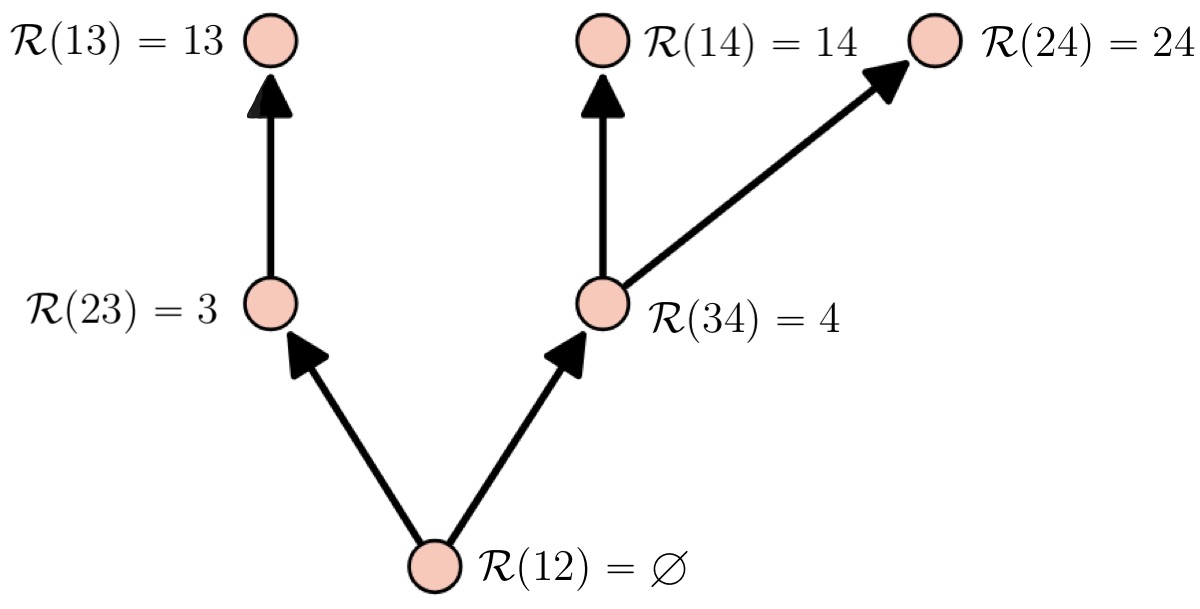}
\end{center}
As can be explicitly checked, this is not the restriction set poset of any shelling order coming from the cube.
\end{ex}

In the proof of Proposition \ref{prop:nonRev} we use three disjoint bases to demonstrate the existence of counterexamples. In general, there are only finitely many matroids with no loops or coloops of a fixed rank for which these three bases fail to exist. This motivates the following question:

\begin{ques}\label{ques:classify-reversible}
Is it possible to classify the matroids for which all shelling orders are reversible? 
\end{ques}

We also believe that the following question will have interesting combinatorial consequences.

\begin{ques}\label{ques:construct-weakly-geometric}
Is there a systematic way to construct shelling orders that are not weakly geometric?
\end{ques}

\section{Line shellings and internal activity}\label{section:line-shellings-and-internal-activity}

In the last section we proved that the polytope $P_M^*$ induces a large collection of shelling orders for $\I(M)$. We now explore what line shellings of $P_M^*$ yield in this setting, beginning with a proof of part A of Theorem \ref{thm:lweights-imply-shelling} which was stated in the introduction.

\begin{proof}[Proof of Theorem \ref{thm:lweights-imply-shelling}, part A]

We already discussed how $\ell$ can be interpreted as a real-valued function on $E$ or as an element of $(\mathbb{R}^E)^*$. The $\ell$ weights are exactly the $\ell$-evaluations on the vertices of the matroid polytope, hence the weight ordering is a line shelling of $P_M^*$. By Theorem \ref{thdm:shell}, this gives a shelling order of $\I(M)$.
\end{proof}

\begin{rem}
By choosing $\ell$ such that  $\{\ell(e)\, | \, e\in E\} = \{2^{|E|} - 2^i \,| \,\ i=1,2,\dots, |E|\}$ the total order on $\B$ induced by $\ell$ is the lexicographic order. This can be thought of as encoding lexicographic order in a binary sequence, and recovers a result of Bj\"orner. Furthermore, every shelling order of this form is reversible (change $\ell$ for $-\ell$). Thus the colex order on $B$ is also a shelling order, recovering a result of Chari \cite{Chari1996ReversibleShellingsAndInequalityForHVectors} and showing Hibi's inequalities (see also \cite{chari2, Hibi}).
\end{rem}

\begin{rem}
One of the prominent features of lexicographic shellings is its relation to Tutte's activity theory and the Tutte polynomial. The restriction sets of these shellings are called the internally passive sets of the underlying order of the ground set. While not important in this article, we remark that external activity, the remaining ingredient to get the whole Tutte Polynomial, is exactly the data corresponding to the lexicographic shelling order of the dual matroid with respect to the same order.
\end{rem}

In order to continue our study of line shellings, the following definition will be quite useful. 

\begin{defn}
Let $\Delta$ be a pure simplicial complex and let $<$ be a shelling order of $\Delta$. The restriction set poset $P(\Delta, <)=(\mathcal{F}, \prec)$ is the poset on the set $\mathcal{F}$ of facets of $\Delta$ where $F \prec F'$ if $\mathcal{R}(F) \subseteq \mathcal{R}(F')$.
\end{defn}

This poset has been studied for ordered matroids before, without making much reference to shellability. If $<$ is an ordering of the ground set $E$ and $<_\text{Lex}$ is the associated lexicographic ordering of the bases, then $P(\I,<_\text{lex})=\text{Int}_<(M)$ as defined by Las Vergnas in \cite{LasVergnas}. Furthermore, this poset showed up in Dawson's work \cite{Dawson} even before the systematic study of Las Vergnas. We state here the main theorems of interest. 

\begin{thm}\label{thm:Dawson-Chari} Let $M=(E,\I)$ be a matroid, $<$ be a total order on $E$ and let $<_\text{lex}$ be the lexicographic order on the bases. Then the following is true for $P:= P(\I, \lex)$
\begin{enumerate}
    \item[A.] \cite{Dawson} The poset $P$ is ranked by the size of the restriction set of the basis and is the containment poset of a greedoid, that is, if $|IP_<(B)|<|IP_<(B')|$, then there is $b\in IP_<(B') \backslash IP_<(B)$ and a basis $B''$ such that $IP_<(B'')= IP_<(B)\cup\{b\}$.\footnote{Notice that this last property resembles that of the independence complex of a matroid. However, the set system of the internally passive sets here is not a simplicial complex. }
    \item[B.] \cite{LasVergnas} Let $\hat P$ be the poset $P$ with an extra element that is a maximum. Then $\hat P$ is a lattice.
\end{enumerate}
\end{thm}

We claim that the restriction set poset associated to a line shelling with functional $\ell$ depends only on the underlying ordering induced by the coordinates, i.e, the braid cone it belongs to. This is related to Theorem~3.8 in \cite{ArdCasSam}: the restriction set poset is exactly $\text{Int}_<(M)$, so it suffices to show that the order by weights is a linear extension. It is now time to prove part B of Theorem \ref{thm:lweights-imply-shelling}.

\begin{proof}[Proof of Theorem \ref{thm:lweights-imply-shelling}, part B]
Notice that $\ell$ induces an acylic orientation on the one skeleton of $P_M$: orient each edge from the smallest to the largest weight of the vertices which are its endpoints. This acyclic orientation endows $\B(M)$ with a well-studied poset structure, the Gale ordering of the bases. The article \cite{Samper} shows that the Gale poset as defined in Section 2 of \cite{Samper} is a coarsening of $\text{Int}_<(M)$. Now notice that ordering the bases according to total weight produces a linear extension of the Gale ordering and hence a linear extension of $\text{Int}_<(M)$. The result follows from Theorem 4.14 in \cite{ArdCasSam}.
\end{proof}

\section{Extendable shellability and heuristics for Simon's conjecture}\label{section:extendable shellability-heuristics-for-simons-conjecture}

As discussed in the introduction, we now relate our results to Simon's conjecture.

\begin{conj}[Simon \cite{Simon}]
For every $1< k \le n$, the independence complex of the uniform matroid $U_{n,k}$ is extendably shellable, i.e if $B_1<\dots <B_s$ are bases of $U_{n,k}$ ordered to shell $\langle B_1,\dots,\, B_s \rangle$, then there is an ordering $B'_1<\dots < B'_t$ of the remaining bases of $U_{n,k}$ such that $$B_1 < \dots B_s < B'_1 < \dots <B'_t$$ is a shelling order of $U_{n,k}$.
\end{conj}

In this section we make two observations that hopefully serve as a map to navigate the properties of partial shellings and help prove or disprove the conjecture. The first one is purely combinatorial and is related to Theorem 1.3 in \cite{ArdCasSam}. We rewrite the theorem here and state a simple corollary on extendable shellability.

\begin{thm} Let $M$ be a matroid and $<$ be a total order on its ground set. Any linear extension of $\text{Int}_<(M)$ is a shelling order of $\I(M)$. 
\end{thm}

As a corollary, we get the following result on extendable shellability. For this, recall that an order ideal $Q$ of a partially ordered set $P$ is a subset of its elements that is downward closed, that is, a subset satisfying that if an element $p$ is in $Q$, then all elements in $P$ below $p$ are also in $Q$.

\begin{cor}\label{cor:extend}
Let $M$ be a matroid and let $B_1,\dots ,B_s$ be a partial shelling such that the set $\{B_1,\dots B_s\}$ is an order ideal of $\Int_<(M)$ for some order $<$ on the groundset of $M$. Then the partial shelling is extendable.
\end{cor}

\begin{proof}
Let $<$ be the witness order of the statement and let $\prec$ be any linear extension of $\Int_<(M)$ such that the first $s$ elements are $\{B_1,\dots, B_s\}$ in some order. Let $B_{s+1},\dots, B_t$ be the remaining bases ordered using $\prec$. Then $B_1,\dots, B_t$ is a shelling order: the shellability property applied to $B_i$ depends only on $\{B_1,\dots, B_{i-1}\}$ as a set and not on how the previous bases were ordered. 
\end{proof}

While there are potentially many partial shellings that do not satisfy the condition of Corollary \ref{cor:extend}, it does not seem easy to construct them. Term orders which typically yield ways of sorting the sets are poised to fail: by Theorem \ref{thm:lweights-imply-shelling} the restriction set poset of a shelling order coming from \textit{any} term order produces the internally passive poset associated to the braid cone of the underlying order used with the term order. In fact, we know that there are matroids for which the independence complex is not extendably shellable. For instance, the boundary of any cross-polytope of dimension $\geq 12$ is the independence complex of a graphic matroid and was shown to not be extendably shellable in \cite{Hall2004CounterexamplesID}.

Thus, while we still agree with the widely spread opinion that the conjecture is false, it is worth remarking that a really novel idea is needed to prove/disprove it. A consequence of Theorem \ref{thdm:shell} is that Simon's Conjecture is to some extent related to the question of extendable shellability of dual hypersimplices. If $M$ is any matroid, then every extendable partial shelling of $P_M^*$ is an extendable partial shelling of $\I(M)$.

We remark that neither problem implies the other: the set of shelling orders of $\I(M)$ is larger than the set of shelling orders of $P_M^*$, thus extendable shellability of $P_M^*$ does not imply extendable shellability of $\I(M)$. The converse is also false. If $\I(M)$ is extendably shellable, then a partial shelling if $P_M^*$ indeed yields a partial shelling of $\I(M)$, but it may be the case that the extensions in $\I(M)$ are not shelling orders of $P_M^*$. We do not know much more about the case of the hypersimplex, which is the one directly connected to Simon's conjecture. 

However, deciding the property for one of the two objects is likely to yield interesting consequences on the other side of the picture. We therefore leave the following conjecture and two questions, all of which are likely to be more tractable than Simon's conjecture.

\noindent \textbf{Conjecture \ref{conj:hypersimplex-extendably-shellable}.}
\emph{The dual polytope of every hypersimplex is extendably shellable.}

\begin{ques}\label{ques:dual-matroid-extendably-shellable}
For which matroids is the dual matroid polytope extendably shellable?
\end{ques}

\begin{ques}\label{ques:boundaries-crosspolytopes}
The boundaries of cross-polytopes of dimensions at least 12 are independence complexes that are not extendably shellable. Do the partial shellings which fail to extend (see \cite{Hall2004CounterexamplesID}) correspond to partial shellings of the dual matroid polytope?
\end{ques}

\section{Pinned broken lines and generalized internal activity}\label{section:pinned-broken-lines-and-generalized-internal-activity}
In \cite{DallInternallyPerfect2017}, Dall proves that for a large class of ordered matroids $(M,<)$, which he calls internally perfect, the restriction set poset $\text{Int}_{<}(M)$ is the divisibility poset of a multicomplex, thereby showing Stanley's pure $\mathcal{O}$-sequence conjecture for this class of matroids.

As seen above, the internal activity comes from line shellings associated to generic linear functionals. A straight-forward way to modify this order is by thinking of the linear functional as sweeping the matroid polytope and slightly wiggling the normal vector defining the hyperplane. Accordingly we prove the following lemma and give a related definition. 

\begin{lemma}\label{lem:broken}
Let $M$ be a matroid and assume that $B_1<B_2<\dots < B_n$ is an ordering of the bases such that for every $i=1,\dots, n$, there exists a linear functional $\ell_i \in (\mathbb{R}^E)^*$ such that $\ell_i(B_j) < \ell_i(B_i)$ if and only if $j<i$. In this case, the order is a shelling order of $\I(M)$ such that the restriction set $\RR(B_i)$ of the basis $B_i$ is $IP_{\ell_i}(B_i)$.
\end{lemma}

\begin{proof}
We use again the fact that checking the shellability condition for $B_i$ only depends on the complex generated by $B_1,\dots, B_{i-1}$ and not on the order they were attached. By Theorem \ref{thm:lweights-imply-shelling} the order $<_{\ell_i}$ on the bases $\{B_1,..., B_i\}$ is a partial shelling of the matroid ending in $B_i$. Thus the shellability condition holds for $B_i$ and the restriction set is the same as in the shelling order $<_{\ell_i}$, that is, equal to $IP_{\ell_i}(B_i)$ by Theorem \ref{thm:lweights-imply-shelling} part B.
\end{proof}

\begin{defn}\label{def:broken-line-shelling-and-witnesses}
Any order satifsfying the conditions of Lemma $\ref{lem:broken}$ is called a \emph{broken line} shelling order. The functionals $\ell_1, \dots, \, \ell_n$ are called \emph{witnesses} of the broken line shelling. 
\end{defn}

Of course, any line shelling is an example of a broken line shelling: the functional $\ell_i$ is independent of $i$ in such a shelling order. Furthermore, notice that the witnesses are far from unique. In Section \ref{section:examples} we will see examples of broken line shelling orders that lead to restriction set posets different from the internally passive poset for any line shelling. The goal of this discussion is to present some ideas directed toward understanding Stanley's pure $\mathcal{O}$-sequence conjecture for matroid independence complexes. 

\begin{defn}
Let $M$ be a matroid, let $\prec$ be a broken line shelling order of $\I(M)$ and let $B$ be a a basis. The $\prec$-internal activity of $B$, denoted by $IP_{\prec}(B)$, is the restriction set of $B$ in the shelling order $<_{\ell_i}$ as in Lemma \ref{lem:broken}. Let $\text{Int}_\prec(M)$ be the poset on the bases of $M$ ordered by inclusion of the internally passive sets coming from $\prec$.
\end{defn}

Our next goal is to try to determine if at least some of these shelling orders give a rich poset structure, perhaps better than that given by line shellings. Recall from Theorem \ref{thm:Dawson-Chari} that for a line shelling, the restriction set admits a greedoid structure and is a graded lattice once an artificial maximum is attached. In the future, we would like to study conditions under which at least one of these properties holds for broken line shellings perhaps with other desirable properties. 

Inspired by Stanley's $h$-vector conjecture, we focus on a special kind of broken line shellings. These have the potential to extend Dall's work on internally perfect matroids \cite{DallInternallyPerfect2017} and prove the main Conjecture 3.10 of \cite{KleeSam2015}, perhaps in some special cases. Our goal is to make sure that the poset is ranked, that the rank of an element of the poset corresponds to the size of the restriction set, and to hope for some additional structure. For instance, a Lemma secretly hidden in the work of Las Vergnas that is essential for Conjecture 3.10 in \cite{KleeSam2015} is the following: 

\begin{lemma}[\cite{LasVergnas} Proposition 2.5]\label{lem:LasVergnas}
Let $M$ be a matroid with groundset $E$, let $\ell$ be a linear functional and let $B_1$ be the basis with the smallest $\ell$-weight. The containment $B\backslash B_1\subseteq IP_\ell(B)$ holds for any basis $B$. 
\end{lemma}

The following simple observation holds for any shelling order of a pure simplicial complex. The proof is straight-forward and therefore omitted. 

\begin{lemma}\label{lem:restrictionVertices}
Let $\Delta$ be a pure simplicial complex and let $<$ be a shelling order whose first facet is $F_0$. Let $v$ be any vertex not in $F_0$ and let $F$ be the first facet of the shelling containing $v$, then $\RR(F) = \{v\}$.
\end{lemma}

Notice that Lemma \ref{lem:LasVergnas} is a generalization of Lemma \ref{lem:restrictionVertices} that does not generalize to arbitrary shelling orders. For instance, the following example shows that even one change of linear functional suffices to lose many of the nice properties. 

\begin{ex}
Consider the uniform matroid $U_{4,2}$ with vertices numbered from $1$ to $4$. The order $12, 13, 23, 34, 24, 14$ is a broken line shelling order, with witness functionals $\ell_1=(1,2,3,8)$ to order the first three bases and $\ell_2 = (3,2,1,8)$ to order the last three. The restriction set poset of this shelling order is a graded meet semi-lattice isomorphic to that of Example \ref{ex:independenceButNotPolytope}. It fails part \emph{A.} of Theorem \ref{thm:Dawson-Chari}. The restriction set of $34$ is $4$, yet the first basis of the shelling order is $12$, so the conclusion of Lemma \ref{lem:LasVergnas} fails for this shelling order. 
\end{ex}

The main issue in the example above is that the smallest bases of each of the two functionals are different. While $\ell_1$ is minimized at the vertex corresponding to the basis $12$, $\ell_2$ is minimized by the vertex corresponding to $23$. Thus Lemma \ref{lem:LasVergnas} does not apply and we see the explicit counterexample. We therefore narrow the study to a slightly smaller, yet very general class of shellings. 

\begin{defn}\label{def:pinned-broken-line-shelling}
Let $M$ be a matroid, let $B_1<B_2<\dots<B_n$ be a broken line shelling. We say that the shelling order is \emph{pinned} if there are witnesses $\ell_1,\,\dots \, \ell_n$ such that the smallest basis for each functional $\ell_i$ is $B_1$. 
\end{defn}

Pinned broken line shellings have more structured restriction set posets: Lemma \ref{lem:LasVergnas} applies and we summarize this in the following lemma, whose proof is a direct combination of Lemmas \ref{lem:broken} and \ref{lem:LasVergnas}. 

\begin{lemma}\label{lem:coveredAtoms}
Let $M$ be a matroid, let $B_1 < B_2 <\dots < B_n$ be a pinned broken line shelling with $\text{Int}_<(M)$ the associated restriction set poset. The atoms of $Int_<(M)$ are in bijection with the elements of $E$ that are neither loops nor contained in $B_1$. Furthermore, for any basis $B$, the atoms of $\text{Int}_<(M)$ below $B$ are exactly those bases whose restriction set is a single element of $B\backslash B_1$.
\end{lemma}

In fact there are a couple of cases that imply that the internal poset is well structured. For instance, it would be interesting to understand under which conditions it admits a greedoid structure, or when it is graded even after attaching an artificial top element. This seems complicated, but many of the examples shown below suggest that both are frequent phenomena for which there may be a good collection of conditions which imply these properties.

\section{A geometric approach to Stanley's pure O-sequence conjecture}\label{section:geometric-approach-to-stanleys-pure-o-sequence-conjecture}
Here we present an idea to systematically find pure multicomplexes that witness the validity of Stanley's conjecture for a given matroid. The idea is to combine the technique of \cite{DallInternallyPerfect2017} but replacing the lexicographic shelling with a pinned broken line shelling. The reason for restricting to broken lines comes from the solution for rank four matroids \cite{KleeSam2015}. There, the suggestion is to construct a multicomplex by gluing together several multicomplexes, and the following $h$-vector identity appears:

\begin{thm}[\cite{KleeSam2015}] Let $M$ be a matroid and let $B$ be any basis, then: 
\[h(M,t) = \sum_{I\in \I, I\cap B= \emptyset} t^{|I|}h((M\slash I)|_B, t)\]
Here, $M\slash I$ is the contraction of the set $I$ and $(M\slash I)|_B$ means the restriction of the matroid to the ground set $B$.
\end{thm}

The heuristic behind the use of this identity is the following: inductively, each of the matroids $(M\slash I)|_B$ would be equipped with a multicomplex that we can use as a building block for $M$. The relationship with the results above is spelled out in the following lemma.

\begin{lemma}
Let $M$ be a matroid and let $<$ be a pinned broken line shelling whose first basis is $B$ and whose induced internal order is denoted by $\prec$. For each non loop element $i\in E\backslash B$ let $B_i$ be the basis whose restriction set with respect to $<$ is $\{i\}$. Let $I$ be an independent set disjoint from $B$. The number of bases $B'$ whose restriction set has $|I|+j$ elements and such that the atoms of $\text{Int}_<(M)$ below $B'$ correspond to the set $\{B_i\, | \, i\in I\}$ is $h_i((M\slash I)|_B)$. If $Int_<(M)$ happens to be ranked, then the cardinality of the restriction sets may be replaced by the rank in $Int_<(M)$.
\end{lemma}

\begin{proof} This is a reinterpretation of Lemma \ref{lem:coveredAtoms}. \end{proof}

In other words, the internal poset of pinned broken shellings keeps track of the decomposition of the $h$-vector as long as it is graded. More informally, the restriction set poset is somewhat similar to a divisibility poset (which we are looking for) and one may wonder if it may very well be. We recall that this idea is not entirely new for the classical theory of internal activity. In fact Dall showed that it works for several ordered matroids, which he dubs internally perfect. With this in mind we formulate the following conjecture, whose feasibility we explore computationally in the following sections. 

\noindent \textbf{Conjecture \ref{conj:friendly-brokenLine}.} \emph{Given a matroid $M$ there exists a pinned broken line shelling order $<$ such that the internal activity poset $\text{Int}_<(M)$ associated to the shelling is the divisibility poset of a pure multicomplex.}

Perhaps the largest piece of evidence for the conjecture is Dall's result \cite{DallInternallyPerfect2017}. For the rest of the section, we discuss some heuristics that explain why this conjecture is reasonable. 

The most complicated aspect of Stanley's conjecture is the purity of the multicomplex: since matroids are shellable, their Stanley-Reisner ring is Cohen-Macaulay and the existence of a multicomplex was shown by Stanley \cite{S1977}. However, the proof uses regular sequence and lexicographic segments, thus the combinatorial structure of the resulting multicomplex has very little to do with that of the matroid. In Dall's theorem, the result follows essentially by the statement that the restriction sets form a Greedoid as a set system, e.g see Theorem \ref{thm:Dawson-Chari}.

Furthermore, notice that changing linear functionals in the broken line shellings takes parts of distinct line shelling posets and merges them in a way that is topologically coherent. The combinatorial data changes, but one might expect that if the change is not too drastic some of the combinatorial properties will be preserved. Thus, the main idea behind the conjecture is the expectation that one can slightly perturb the line shelling in a way that the new combinatorics preserves greedoid-like properties (like gradedness), but in such a way that the new partially ordered set is a divisibility poset.

The conjecture seems far fetched, and we expect it to be quite difficult. However, it is likely to lead to new results in special cases. In Section \ref{section:examples} we present some examples of ordered matroids that can be easily perturbed to obtain solutions to the conjecture. Our experiments show that, for a given matroid, there are many restriction set posets of pinned broken line shellings that do not appear as classical internal activity posets. 

Given a broken line shelling, each of the contracted elements in the recursion corresponds to a face of the matroid polytope. The broken line shelling is also a broken line shelling of each of the smaller pieces. Therefore we may hope to perform the inductive technique from the conjecture in \cite{KleeSam2015} by constructing broken line shellings whose restriction to each of these faces is well-behaved.

We also remark that any shelling order of $\I(M)$ or $P_M^*$ induces an orientation of the graph of $P_M$. The orientation is acyclic, and thus yields a partially ordered set structure on the bases. In the case of line shellings, this order is commonly called the Gale order of the bases and it is immediate that the shelling order is a linear extension of said poset. Thus the following definition is in place: 

\begin{defn}
Let $M$ be a matroid and let $<$ be a broken line shelling. Let $G_<(M)$ be the orientation of the graph of $P_M$ induced by $<$ and let $\text{Gale}_<(M)$ be the poset induced by $G_<(M)$. 
\end{defn}

The shelling order $<$ is a linear extension of $\text{Gale}_<(M)$ and the structures of these posets may hold the key to understanding differences between shelling orders. In particular, it would be of interest to find combinatorial conditions on the graph of $P_M$ that guarantee the existence of broken line shelling orders inducing them. We end with a series of questions that may lead to better understanding of the internal activity of a broken line shelling.

\begin{ques}\label{ques:acyclic-orientation}
Are there combinatorial conditions that guarantee that an acyclic orientation of the graph of $P_M$ is induced by a shelling order? How about broken line shelling orders? How about pinned broken line shelling orders?
\end{ques}

\section{Broken line shellings in SAGE}\label{section:implementation-SAGE}

The purpose of this section is to describe how to use the exploratory software we developed for broken line shellings. The code is freely available at \cite{brokenlinegithub}. Before giving the basics in Subsection \ref{subsection:basic-usage}, we make a few helpful definitions. In the interest of describing the explicit computations our software makes, we change terminology slightly. Definition \ref{def:pinned-broken-line-shelling} refers to a \textit{pinned broken line shelling}. All our broken line shellings will be pinned, and we will have specific \textit{witnesses} as in Definition \ref{def:broken-line-shelling-and-witnesses}, so we will call such a pinned broken line shelling with specific witnesses a \textit{sweep}. We give this definition precisely in Definition \ref{def:sweep}.

First, given a matroid polytope $P_M$ let $V$ be the vertex set of the polytope. For this section, we entirely focus on the polytope and so we will forget most of the other information. For the purposes of describing the algorithms and computer code, we change the notation slightly from above. In particular, we need to assign various linear functionals to vertices, but also evaluate linear functionals on vertices. Therefore we will use $\ell^T v_i$ for the evaluation and $\ell(v_i)$ for an assignment. We will explain this now. We define a hyperplane arrangement associated to $P_M \subset \mathbb{R}^E$. For each pair of vertices $\{v_i, v_j\} \in \binom{V}{2}$ we consider the hyperplane
\begin{equation*}
    H_{ij} = \{ \ell \in (\mathbb{R}^E)^* \, : \, \ell^T(v_i - v_j) = 0 \}
\end{equation*}
of all linear functionals $\ell \in (\mathbb{R}^E)^*$ which agree on the vertices $v_i$ and $v_j$. Let $\mathcal{H}$ be the hyperplane arrangement with the hyperplanes above. Here the natural pairing between $\mathbb{R}^E$ and $(\mathbb{R}^E)^*$ is expressed by $\ell^T v$, which is equivalent to evaluating the linear functional $\ell$ at the point $v \in \mathbb{R}^E$. The collection of all such $\binom{V}{2}$ hyperplanes stratifies $(\mathbb{R}^E)^*$ into open regions of functionals $\ell$ which assume distinct values $\ell^T v_i \neq \ell^T v_j$ on all vertices of $P_M$. Notice that this is a refinement of the normal fan. For instance, there are hyperplanes of functionals agreeing on vertices not connected by an edge in $P_M$. Such linear functionals induce an ordering of the vertices of the matroid polytope, given by $v_i < v_j$ if and only if $\ell^T v_i < \ell^T v_j$ as in Lemma \ref{lem:broken}. Given a linear functional $\ell$ and a vertex $v$, we also have a partition of the vertex set into the vertices $w$ such that $\ell^T w \leq \ell^T v$, and the vertices $u$ such that $\ell^T u > \ell^T v$. We say that $\ell_1$ and $\ell_2$ \textit{induce the same cut} of the matroid polytope at the vertex $v$ if these partitions are equal.
\begin{defn}\label{def:sweep}
    Consider a matroid polytope with vertex set $V$ of size $n$, with the matroid ground set denoted $E$. A \textit{sweep based at $v_1 \in V$} is a function $\ell:V \to (\mathbb{R}^E)^*$ from the set of vertices to the set of linear functionals, and a labelling of the vertices $(v_1,v_2,\dots,v_n)$ satisfying the requirements
\begin{enumerate}
    \item Every linear functional $\ell(v_k)$ is minimized at the vertex $v_1$.
    \item At each $k \in [n-1]$ the linear functional $\ell(v_k)$ induces the cut $\{v_1,\dots,v_k\} \cup \{v_{k+1},\dots,v_n\}$ of the matroid polytope at $v_k$.
    \item Each linear functional $\ell(v_k)$ takes distinct values on each of the vertices of the matroid polytope. This will be true for generic linear functionals.
\end{enumerate}
\end{defn}

There is a natural equivalence relation on all such sweeps $\ell:V \to (\mathbb{R}^E)^*$, $(v_1,\dots,v_n)$ satisfying the requirements above. If $\ell_1(v_k)$ and $\ell_2(v_k)$ lie in the same region of the hyperplane arrangement $\mathcal{H}$ for every $k \in [n], n = |V|$, we identify the two sweeps. Notice that there are different equivalence classes that induce the same shelling order. However, due to the inductive nature of this construction, we find it useful to keep track of the refined chambers in order to better keep track of the geometry.

You might first imagine simply enumerating all possible sweeps, and trying them all, and then sorting through all the restriction set posets. However, this is computationally infeasible. We have found that new and interesting posets can be created simply by pivoting at a few vertices, searching for a few new linear functionals at each pivot, and then sorting through the posets obtained in this way afterwards. We give examples in Section \ref{section:examples}. The software we demonstrate below in Subsection \ref{subsection:basic-usage} proceeds roughly as follows: We will construct the sequence of $\ell(v_1), \ell(v_2), \dots$ one step at a time. The default is to use a line shelling, but the user can specify desired pivots at certain locations, creating a broken line shelling.

\begin{enumerate}
    \item Given the matroid polytope and its vertices, we first calculate normal vectors for each of the hyperplanes of linear functionals which agree on some pair of vertices.
    \begin{align*}
        \text{normals } &= \{ v_i - v_j \}_{ij \in \binom{V}{2}}\\
        \text{hyperplanes } &= \bigcup_{ij \in \binom{V}{2}} \{ \ell \in (\mathbb{R}^E)^* \, : \, \ell^T(v_i - v_j) = 0 \}
    \end{align*}
    This hyperplane arrangement splits the dual space into regions of linear functionals which produce distinct orderings of the vertices. We fix an ordering of these normal vectors arbitrarily $(n_1,n_2,\dots,n_M)$, since later we will characterize an open region of linear functionals by the sequence of $+1$ and $-1$ created by the signs of the dot products against this ordered list of normal vectors.
    \item We begin by finding a linear functional which minimizes \texttt{vFav} $= v_1$. We sample uniformly from the sphere $\mathbb{S}^{|E|-1} \subset (\mathbb{R}^E)^*$ by independent normal distributions in each coordinate, normalizing to unit length. We do this until we find a linear functional which minimizes $v_1 =$ \texttt{vFav}. We also project this vector along the all ones vector, since the matroid polytope lies inside the affine hyperplane whose normal vector is a multiple of the all ones vector, shifted off the origin according to the rank of the matroid.
    \item We proceed to search for and record each new witness $\ell(v_{j+1})$. If the user has not specified a pivot at position $j$ then we simply copy $\ell(v_j)$ to $\ell(v_{j+1}) := \ell(v_j)$. However, if the user has specified a pivot at position $j$, then we produce new possible $\ell(v_{j+1})$ linear functionals by random sampling, each time checking if the conditions of Definition \ref{def:sweep} are met. We do this by taking a random convex combination
    \begin{equation*}
        \ell(v_{j+1}) = \sum_{i=2}^n c_i (v_i - v_1) \text{ such that } \sum c_i = 1
    \end{equation*}
    where the $c_i$ are chosen randomly from the interval $(0,1) \subset \mathbb{R}$ and then normalized to have $\sum c_i = 1$. This has the added benefit of \textit{weighting our search} to sample more densely from normal vectors which have more vertices in that general direction. This seems to produce more interesting posets in less random samples. We can imagine other ways to choose new functionals, and encourage users to experiment and change the code we provide.
    \item Also, the user can specify a parameter \texttt{w} which changes the way we produce a new $\ell(v_{j+1})$. Setting $w=0$ causes $\ell(v_{j+1})$ to be the random convex combination $\sum c_i (v_i - v_1)$. But with nonzero $w$, we form $\ell(v_{j+1})$ by first taking $w \cdot \ell(v_j)$ and then adding to it the random convex combination $\sum c_i (v_i - v_1)$. This has the effect of shifting the cone forward along $w \cdot \ell(v_j)$, producing a more localized sample as in Figure \ref{fig:shifted-cone}.
    \begin{figure}[H]
        \centering
        \includegraphics[width=0.4\textwidth]{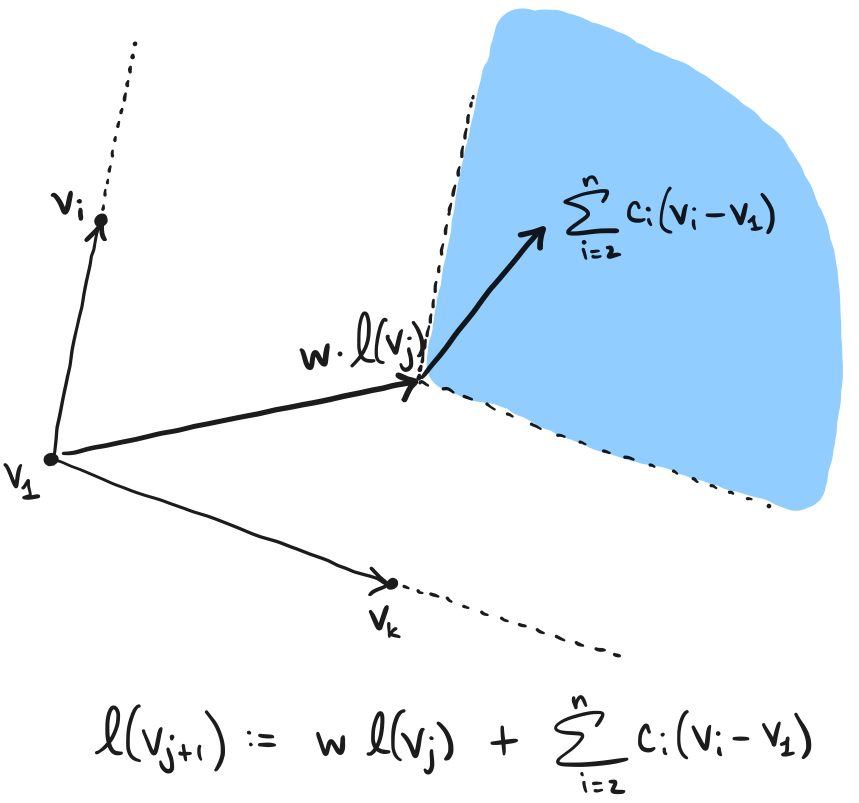}
        \caption{Using the parameter w}
        \label{fig:shifted-cone}
    \end{figure}
    The end result of using some $w > 0$ is that $\ell(v_{j+1})$ will be more likely to point in roughly the same direction as $\ell(v_j)$. As $w \to 0$ the dependence on the previous $\ell(v_j)$ disappears.
\end{enumerate}

\subsection{Basic usage}\label{subsection:basic-usage}

The main functions available to the user are the following, whose parameters we will explain below:
\begin{verbatim}
        result = search(V, vFav, pivots, limit, misses, w)
        display_results(result)
\end{verbatim}
As the user varies the parameters \texttt{pivots, limit, misses, w} and runs the search many times with different parameter choices, it is beneficial to use the commands
\begin{verbatim}
        result = update_search(result, V, vFav, pivots, limit, misses, w)
        display_results(result)
\end{verbatim}
since this keeps any posets obtained from previous searches, and simply updates the list of distinct posets obtained. This allows much experimentation, with each new run of \texttt{update\_search} only potentially increasing your results, rather than losing previous posets obtained in some other sweep. The contents of \texttt{result} will always remember the exact sweep of linear functionals, and so no information is lost by changing parameters and experimenting.

The inputs to the function \texttt{search} are \texttt{V, vFav, pivots, limit, misses, w}, which we describe below. The inputs to the function \texttt{update\_search} are the same, except the first positional argument is a \texttt{result} of a previous \texttt{search}, and then the following positional arguments are those we describe now, which are the same for both \texttt{search} and \texttt{update\_search}.
\begin{enumerate}
    \item \texttt{V} is a list containing the vertices of the matroid polytope as tuples of zeros and ones, the characteristic vector $\chi_B$ of some basis $B \in \B$. For example, this can be obtained for many built-in matroids by running the following code in SAGE:
    \begin{verbatim}
    U24 = matroids.Uniform(2,4)
    P = U24.matroid_polytope()
    V = P.vertices_list() \end{verbatim}
    Then \texttt{V} could be used as input to \texttt{search}. Alternatively, you could directly type the following:
    \begin{verbatim}
    V = [[0, 0, 1, 1], [0, 1, 0, 1], [0, 1, 1, 0],
          [1, 0, 0, 1], [1, 0, 1, 0], [1, 1, 0, 0]]
    \end{verbatim}
    \item \texttt{vFav} is your favorite vertex, given as a list or tuple of coordinates. All the linear functionals computed will attain their minimum value on this vertex. Hence, your sweep of the matroid polytope will start at this vertex. For example:
    \begin{verbatim}
        vFav = [0,1,0,1]
    \end{verbatim}
    \item \texttt{pivots} is a list of the locations where you want to pivot linear functionals. The first place you can pivot is location \texttt{1}, for example
    \begin{verbatim}
        pivots = [1,3,4]
    \end{verbatim}
    would cause the program to generate new linear functionals $\ell(v_2), \ell(v_4)$ and $\ell(v_5)$, since we are pivoting \textit{after} vertices $1,3,4$. Elsewhere, the same linear functional is simply copied forward, so that $\ell(v_{j+1}) := \ell(v_j)$ for each $j \notin \texttt{pivots}$.
    \item \texttt{limit} is the maximum number of new $\ell(v_{j+1})$ to produce at each pivot. As the total number of sweeps produced may be as high as $\texttt{limit}^{|\texttt{pivots}|}$, it is advisable to set this relatively low as you begin searching for new posets. Otherwise the program will run for too long. For the same reason, it is even more important to keep $|\texttt{pivots}|$ small, at least to start. We recommend pivoting at one location to start, and also two pivots seems to run quite fast. Once you try three pivots, the program runs noticeably longer.
    \item \texttt{misses}: At each pivot we generate new $\ell(v_{j+1})$ randomly, but we must then check that they induce the same cut of the matroid polytope, and that they are in a distinct open region of the hyperplane arrangement minimizing \texttt{vFav}. Thus, it may be the case that we search for a long time before a new, valid $\ell(v_{j+1})$ is actually produced. The value of \texttt{misses} is the number of failures at which the user would like to give up on that pivot. Thus, the higher the value of \texttt{misses}, the longer it may search at each pivot location, and the longer the program may run.
    \item \texttt{w} is the weight given to $\ell(v_j)$ in producing the next $\ell(v_{j+1}) := w \cdot \ell(v_j) + \sum_{i=2}^n c_i (v_i - v_1)$. This gives a dependency on the previous linear functional, making any tilting of the hyperplane less extreme.
\end{enumerate}

Therefore, a typical run might look like
\begin{verbatim}
        load("shellings.sage")
        V = [(1,1,1,0,0,0),
            (1,1,0,1,0,0),
            (1,1,0,0,1,0),
            (1,1,0,0,0,1),
            (1,0,1,1,0,0),
            (1,0,1,0,1,0),
            (1,0,1,0,0,1),
            (1,0,0,1,1,0),
            (1,0,0,1,0,1),
            (0,1,1,1,0,0),
            (0,1,1,0,1,0),
            (0,1,1,0,0,1),
            (0,1,0,1,1,0),
            (0,1,0,1,0,1)]
        vFav = (1,1,1,0,0,0)
        pivots = [3,6,7]
        misses = 50
        limit = 3
        w = 5.0
        result = search(V, vFav, pivots, limit, misses, w)
        display_results(result)
\end{verbatim}
After the results come back, you may want to change parameters and experiment, but you don't want to lose the posets you have just found. Therefore you run the following code:
\begin{verbatim}
        pivots = [4,5]
        misses = 2000
        result = update_search(result, V, vFav, pivots, limit, misses, w)
        display_results(result)
\end{verbatim}
To keep results, continue using \texttt{update\_search} for further experiments.

\section{Examples}\label{section:examples}

In this section we briefly illustrate how broken line shellings can produce new posets, in particular posets which arise as the divisibility poset of a pure multicomplex, as in Stanley's conjecture. We demonstrate on two examples. First,

\subsection{A graphical matroid}

\begin{center}
    \includegraphics[width=0.35\textwidth]{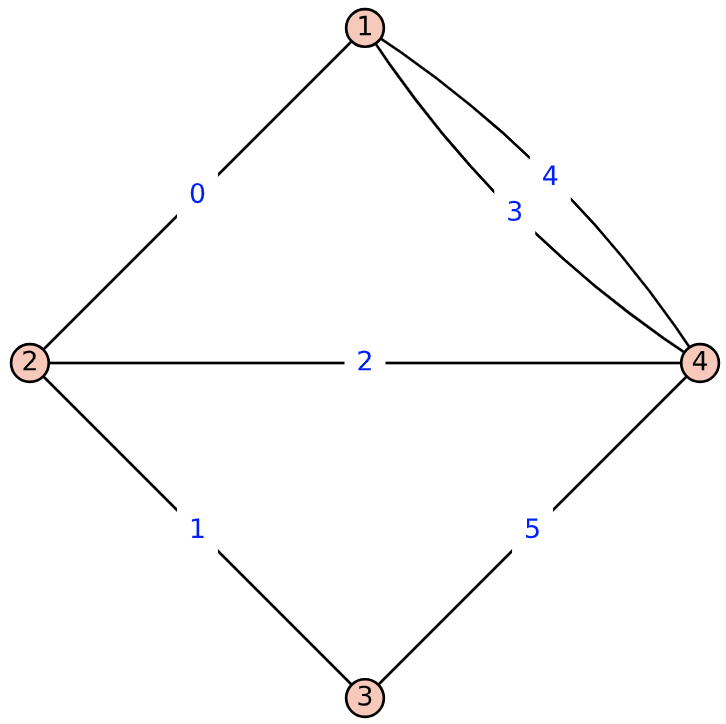}
\end{center}

The matroid whose bases are the 13 spanning trees of this graph has a matroid polytope which has 13 vertices living in $\mathbb{R}^6$. Since Python is zero-indexed, the edges are labelled $0,1,2,3,4,5$, which also becomes the ground set for the matroid. These labels will appear in the internally passive sets below (IP sets). In this example, using only one linear functional (without pivoting) results in a poset which is not the divisibility poset of a multicomplex. However, by using a broken line shelling order and pivoting linear functionals as we sweep through the matroid polytope, we can produce different internally passive sets, which produce many different posets.

The table in Figure \ref{fig:table_initial_bad_poset} shows the construction of IP sets by a sequence of linear functionals which \textit{does not pivot}. In particular the linear functional is $(1,2,3,4,5,6)$ where these are the values of the functional on the six coordinates. Since we use the same functional for the entire sweep of the matroid polytope, this corresponds to a line shelling. The resulting IP sets are given, as well as the order in which this sequence of linear functionals orders the vertices.
\begin{figure}[H]
    \centering
    \begin{tabular}{llll}
vertex & order swept & IP set & linear functional \\ \hline
$\left(1, 1, 1, 0, 0, 0\right)$ & $0$ & $\left[\right]$ & $\left(1.00,\,2.00,\,3.00,\,4.00,\,5.00,\,6.00\right)$ \\
$\left(1, 1, 0, 1, 0, 0\right)$ & $1$ & $\left[3\right]$ & $\left(1.00,\,2.00,\,3.00,\,4.00,\,5.00,\,6.00\right)$ \\
$\left(1, 1, 0, 0, 1, 0\right)$ & $2$ & $\left[4\right]$ & $\left(1.00,\,2.00,\,3.00,\,4.00,\,5.00,\,6.00\right)$ \\
$\left(1, 1, 0, 0, 0, 1\right)$ & $3$ & $\left[5\right]$ & $\left(1.00,\,2.00,\,3.00,\,4.00,\,5.00,\,6.00\right)$ \\
$\left(0, 1, 1, 1, 0, 0\right)$ & $4$ & $\left[2, 3\right]$ & $\left(1.00,\,2.00,\,3.00,\,4.00,\,5.00,\,6.00\right)$ \\
$\left(1, 0, 1, 0, 0, 1\right)$ & $5$ & $\left[2, 5\right]$ & $\left(1.00,\,2.00,\,3.00,\,4.00,\,5.00,\,6.00\right)$ \\
$\left(0, 1, 1, 0, 1, 0\right)$ & $6$ & $\left[2, 4\right]$ & $\left(1.00,\,2.00,\,3.00,\,4.00,\,5.00,\,6.00\right)$ \\
$\left(1, 0, 0, 1, 0, 1\right)$ & $7$ & $\left[3, 5\right]$ & $\left(1.00,\,2.00,\,3.00,\,4.00,\,5.00,\,6.00\right)$ \\
$\left(1, 0, 0, 0, 1, 1\right)$ & $8$ & $\left[4, 5\right]$ & $\left(1.00,\,2.00,\,3.00,\,4.00,\,5.00,\,6.00\right)$ \\
$\left(0, 1, 0, 1, 0, 1\right)$ & $9$ & $\left[1, 3, 5\right]$ & $\left(1.00,\,2.00,\,3.00,\,4.00,\,5.00,\,6.00\right)$ \\
$\left(0, 1, 0, 0, 1, 1\right)$ & $10$ & $\left[1, 4, 5\right]$ & $\left(1.00,\,2.00,\,3.00,\,4.00,\,5.00,\,6.00\right)$ \\
$\left(0, 0, 1, 1, 0, 1\right)$ & $11$ & $\left[2, 3, 5\right]$ & $\left(1.00,\,2.00,\,3.00,\,4.00,\,5.00,\,6.00\right)$ \\
$\left(0, 0, 1, 0, 1, 1\right)$ & $12$ & $\left[2, 4, 5\right]$ & $\left(1.00,\,2.00,\,3.00,\,4.00,\,5.00,\,6.00\right)$ \\
\end{tabular}
    \caption{Line shelling for the graphical matroid}
    \label{fig:table_initial_bad_poset}
\end{figure}

The poset created from these IP sets is displayed in Figure \ref{fig:initial_bad_poset}, where the nodes are labelled according to the \textit{order swept} column in the table. The nodes labelled $11$ and $12$ cause problems when trying to label this poset with monomials respecting divisibility.
\begin{figure}[H]
    \centering
    \includegraphics[width=0.35\textwidth]{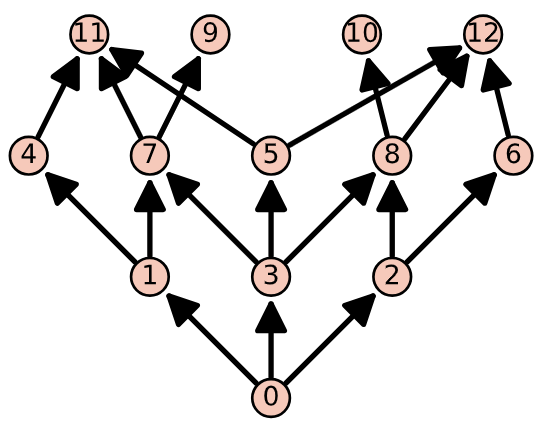}
    \caption{Poset obtained by line shelling}
    \label{fig:initial_bad_poset}
\end{figure}

The table in Figure \ref{fig:table_with_pivoting_better_poset} shows the construction of IP sets by a sequence of pivoting linear functionals, a broken line shelling. By pivoting as we sweep through the matroid polytope, we can produce many new posets. Here, we pivot after the first vertex and after the second (labelled 0 and 1), for a total of three different linear functionals. At each pivot, we check all conditions in Definition \ref{def:sweep}.

\begin{figure}[H]
    \centering
    \begin{tabular}{llll}
vertex & order swept & IP set & linear functional \\ \hline
$\left(1, 1, 1, 0, 0, 0\right)$ & $0$ & $\left[\right]$ & $\left(1.00,\,2.00,\,3.00,\,4.00,\,5.00,\,6.00\right)$ \\
$\left(1, 0, 1, 0, 0, 1\right)$ & $1$ & $\left[5\right]$ & $\left(-6.54,\,4.96,\,4.39,\,6.94,\,6.53,\,5.05\right)$ \\
$\left(1, 1, 0, 0, 0, 1\right)$ & $2$ & $\left[1, 5\right]$ & $\left(-7.78,\,5.30,\,4.18,\,7.94,\,8.00,\,5.91\right)$ \\
$\left(1, 1, 0, 1, 0, 0\right)$ & $3$ & $\left[3\right]$ & $\left(-7.78,\,5.30,\,4.18,\,7.94,\,8.00,\,5.91\right)$ \\
$\left(1, 1, 0, 0, 1, 0\right)$ & $4$ & $\left[4\right]$ & $\left(-7.78,\,5.30,\,4.18,\,7.94,\,8.00,\,5.91\right)$ \\
$\left(1, 0, 0, 1, 0, 1\right)$ & $5$ & $\left[3, 5\right]$ & $\left(-7.78,\,5.30,\,4.18,\,7.94,\,8.00,\,5.91\right)$ \\
$\left(1, 0, 0, 0, 1, 1\right)$ & $6$ & $\left[4, 5\right]$ & $\left(-7.78,\,5.30,\,4.18,\,7.94,\,8.00,\,5.91\right)$ \\
$\left(0, 1, 1, 1, 0, 0\right)$ & $7$ & $\left[2, 3\right]$ & $\left(-7.78,\,5.30,\,4.18,\,7.94,\,8.00,\,5.91\right)$ \\
$\left(0, 1, 1, 0, 1, 0\right)$ & $8$ & $\left[2, 4\right]$ & $\left(-7.78,\,5.30,\,4.18,\,7.94,\,8.00,\,5.91\right)$ \\
$\left(0, 0, 1, 1, 0, 1\right)$ & $9$ & $\left[2, 3, 5\right]$ & $\left(-7.78,\,5.30,\,4.18,\,7.94,\,8.00,\,5.91\right)$ \\
$\left(0, 0, 1, 0, 1, 1\right)$ & $10$ & $\left[2, 4, 5\right]$ & $\left(-7.78,\,5.30,\,4.18,\,7.94,\,8.00,\,5.91\right)$ \\
$\left(0, 1, 0, 1, 0, 1\right)$ & $11$ & $\left[1, 3, 5\right]$ & $\left(-7.78,\,5.30,\,4.18,\,7.94,\,8.00,\,5.91\right)$ \\
$\left(0, 1, 0, 0, 1, 1\right)$ & $12$ & $\left[1, 4, 5\right]$ & $\left(-7.78,\,5.30,\,4.18,\,7.94,\,8.00,\,5.91\right)$ \\
\end{tabular}
    \caption{Broken line shelling for graphical matroid}
    \label{fig:table_with_pivoting_better_poset}
\end{figure}

The poset created from these IP sets is displayed in the following Figure \ref{fig:better_poset_with_pivoting}, where the nodes are labelled according to the \textit{order swept} column in the table above. As you can see, there is no trouble labelling this poset with monomials so that the poset structure corresponds to divisibility.
\begin{figure}[H]
    \centering
    \includegraphics[width=0.35\textwidth]{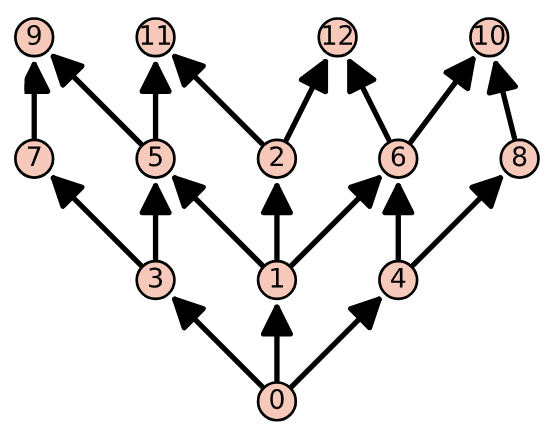}
    \caption{Poset obtained by broken line shelling}
    \label{fig:better_poset_with_pivoting}
\end{figure}

Of course, using a broken line shelling order does not guarantee that the poset you get is any better. Here in Figure \ref{fig:table_bad_poset_despite_pivoting}, for example, is another sequence of linear functionals which give a \textit{bad poset}, isomorphic to the one from before.

\begin{figure}[H]
    \centering
    \begin{tabular}{llll}
vertex & order swept & IP set & linear functional \\ \hline
$\left(1, 1, 1, 0, 0, 0\right)$ & $0$ & $\left[\right]$ & $\left(1.00,\,2.00,\,3.00,\,4.00,\,5.00,\,6.00\right)$ \\
$\left(1, 1, 0, 0, 0, 1\right)$ & $1$ & $\left[5\right]$ & $\left(-0.243,\,2.63,\,5.57,\,11.7,\,11.6,\,6.10\right)$ \\
$\left(1, 1, 0, 0, 1, 0\right)$ & $2$ & $\left[4\right]$ & $\left(1.86,\,-7.10,\,3.30,\,12.7,\,9.05,\,4.34\right)$ \\
$\left(0, 1, 1, 0, 1, 0\right)$ & $3$ & $\left[2, 4\right]$ & $\left(1.86,\,-7.10,\,3.30,\,12.7,\,9.05,\,4.34\right)$ \\
$\left(0, 1, 0, 0, 1, 1\right)$ & $4$ & $\left[4, 5\right]$ & $\left(1.86,\,-7.10,\,3.30,\,12.7,\,9.05,\,4.34\right)$ \\
$\left(1, 1, 0, 1, 0, 0\right)$ & $5$ & $\left[3\right]$ & $\left(1.86,\,-7.10,\,3.30,\,12.7,\,9.05,\,4.34\right)$ \\
$\left(0, 1, 1, 1, 0, 0\right)$ & $6$ & $\left[2, 3\right]$ & $\left(1.86,\,-7.10,\,3.30,\,12.7,\,9.05,\,4.34\right)$ \\
$\left(1, 0, 1, 0, 0, 1\right)$ & $7$ & $\left[2, 5\right]$ & $\left(1.86,\,-7.10,\,3.30,\,12.7,\,9.05,\,4.34\right)$ \\
$\left(0, 1, 0, 1, 0, 1\right)$ & $8$ & $\left[3, 5\right]$ & $\left(1.86,\,-7.10,\,3.30,\,12.7,\,9.05,\,4.34\right)$ \\
$\left(1, 0, 0, 0, 1, 1\right)$ & $9$ & $\left[0, 4, 5\right]$ & $\left(1.86,\,-7.10,\,3.30,\,12.7,\,9.05,\,4.34\right)$ \\
$\left(0, 0, 1, 0, 1, 1\right)$ & $10$ & $\left[2, 4, 5\right]$ & $\left(1.86,\,-7.10,\,3.30,\,12.7,\,9.05,\,4.34\right)$ \\
$\left(1, 0, 0, 1, 0, 1\right)$ & $11$ & $\left[0, 3, 5\right]$ & $\left(1.86,\,-7.10,\,3.30,\,12.7,\,9.05,\,4.34\right)$ \\
$\left(0, 0, 1, 1, 0, 1\right)$ & $12$ & $\left[2, 3, 5\right]$ & $\left(1.86,\,-7.10,\,3.30,\,12.7,\,9.05,\,4.34\right)$ \\
\end{tabular}
    \caption{Another broken line shelling with worse properties}
    \label{fig:table_bad_poset_despite_pivoting}
\end{figure}

\begin{figure}[H]
    \centering
    \includegraphics[width=0.35\textwidth]{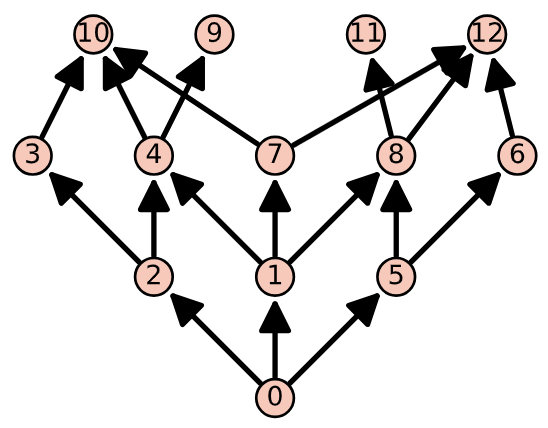}
    \caption{Poset using the broken line shelling of Figure \ref{fig:table_bad_poset_despite_pivoting}}
    \label{fig:bad_poset_despite_pivoting}
\end{figure}

The poset created from these pivots is displayed in Figure \ref{fig:bad_poset_despite_pivoting}. This concludes the example. As you can see, broken line shellings allow the construction of more posets with potentially better properties.

\subsection{The Catalan matroid}

Consider the rank three Catalan matroid. Its matroid polytope has $14$ vertices. We can sweep through the polytope with one linear functional, corresponding to a line shelling. This produces the data in Figure \ref{fig:table_bad_catalan}, which gives the poset in Figure \ref{fig:bad_poset_catalan}.

\begin{figure}[H]
    \centering
    \begin{tabular}{llll}
vertex & order swept & IP set & linear functional \\ \hline
$\left(1, 1, 1, 0, 0, 0\right)$ & $0$ & $\left[\right]$ & $\left(1.00,\,2.00,\,3.00,\,4.00,\,5.00,\,6.00\right)$ \\
$\left(1, 1, 0, 1, 0, 0\right)$ & $1$ & $\left[3\right]$ & $\left(1.00,\,2.00,\,3.00,\,4.00,\,5.00,\,6.00\right)$ \\
$\left(1, 1, 0, 0, 1, 0\right)$ & $2$ & $\left[4\right]$ & $\left(1.00,\,2.00,\,3.00,\,4.00,\,5.00,\,6.00\right)$ \\
$\left(1, 0, 1, 1, 0, 0\right)$ & $3$ & $\left[2, 3\right]$ & $\left(1.00,\,2.00,\,3.00,\,4.00,\,5.00,\,6.00\right)$ \\
$\left(1, 1, 0, 0, 0, 1\right)$ & $4$ & $\left[5\right]$ & $\left(1.00,\,2.00,\,3.00,\,4.00,\,5.00,\,6.00\right)$ \\
$\left(1, 0, 1, 0, 1, 0\right)$ & $5$ & $\left[2, 4\right]$ & $\left(1.00,\,2.00,\,3.00,\,4.00,\,5.00,\,6.00\right)$ \\
$\left(0, 1, 1, 1, 0, 0\right)$ & $6$ & $\left[1, 2, 3\right]$ & $\left(1.00,\,2.00,\,3.00,\,4.00,\,5.00,\,6.00\right)$ \\
$\left(1, 0, 1, 0, 0, 1\right)$ & $7$ & $\left[2, 5\right]$ & $\left(1.00,\,2.00,\,3.00,\,4.00,\,5.00,\,6.00\right)$ \\
$\left(1, 0, 0, 1, 1, 0\right)$ & $8$ & $\left[3, 4\right]$ & $\left(1.00,\,2.00,\,3.00,\,4.00,\,5.00,\,6.00\right)$ \\
$\left(0, 1, 1, 0, 1, 0\right)$ & $9$ & $\left[1, 2, 4\right]$ & $\left(1.00,\,2.00,\,3.00,\,4.00,\,5.00,\,6.00\right)$ \\
$\left(1, 0, 0, 1, 0, 1\right)$ & $10$ & $\left[3, 5\right]$ & $\left(1.00,\,2.00,\,3.00,\,4.00,\,5.00,\,6.00\right)$ \\
$\left(0, 1, 1, 0, 0, 1\right)$ & $11$ & $\left[1, 2, 5\right]$ & $\left(1.00,\,2.00,\,3.00,\,4.00,\,5.00,\,6.00\right)$ \\
$\left(0, 1, 0, 1, 1, 0\right)$ & $12$ & $\left[1, 3, 4\right]$ & $\left(1.00,\,2.00,\,3.00,\,4.00,\,5.00,\,6.00\right)$ \\
$\left(0, 1, 0, 1, 0, 1\right)$ & $13$ & $\left[1, 3, 5\right]$ & $\left(1.00,\,2.00,\,3.00,\,4.00,\,5.00,\,6.00\right)$ \\
\end{tabular}
    \caption{Line shelling for Catalan matroid}
    \label{fig:table_bad_catalan}
\end{figure}

\begin{figure}[H]
    \centering
    \includegraphics[width=0.35\textwidth]{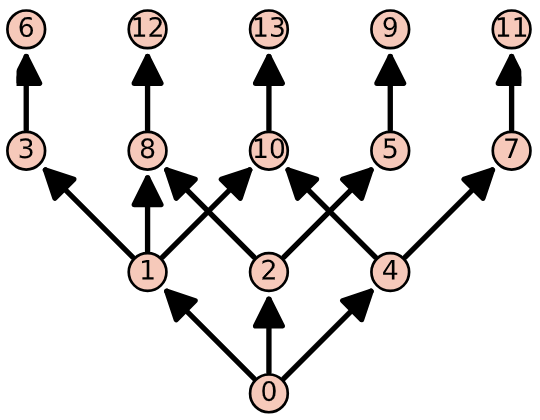}
    \caption{Poset from line shelling for Catalan matroid}
    \label{fig:bad_poset_catalan}
\end{figure}

Instead, if we pivot 4 times, which means we sweep our matroid polytope using 5 different linear functionals, we can create the data in Figure \ref{fig:table_catalan_good_poset}. This produces the poset in Figure \ref{fig:poset_catalan_good}, which is the divisibility poset of a pure multicomplex.

\begin{figure}[H]
    \centering
    \begin{tabular}{llll}
vertex & order swept & IP set & linear functional \\ \hline
$\left(1, 1, 1, 0, 0, 0\right)$ & $0$ & $\left[\right]$ & $\left(1.00,\,2.00,\,3.00,\,4.00,\,5.00,\,6.00\right)$ \\
$\left(1, 0, 1, 1, 0, 0\right)$ & $1$ & $\left[3\right]$ & $\left(-10.7,\,-5.84,\,-14.9,\,2.67,\,18.9,\,17.5\right)$ \\
$\left(0, 1, 1, 1, 0, 0\right)$ & $2$ & $\left[1, 3\right]$ & $\left(0.573,\,3.24,\,-8.68,\,6.11,\,10.6,\,12.3\right)$ \\
$\left(1, 1, 0, 1, 0, 0\right)$ & $3$ & $\left[0, 1, 3\right]$ & $\left(1.32,\,1.56,\,-0.157,\,4.38,\,11.0,\,6.36\right)$ \\
$\left(0, 1, 1, 0, 0, 1\right)$ & $4$ & $\left[5\right]$ & $\left(1.81,\,0.880,\,-0.449,\,4.08,\,10.9,\,6.87\right)$ \\
$\left(1, 0, 1, 0, 0, 1\right)$ & $5$ & $\left[0, 5\right]$ & $\left(1.81,\,0.880,\,-0.449,\,4.08,\,10.9,\,6.87\right)$ \\
$\left(1, 1, 0, 0, 0, 1\right)$ & $6$ & $\left[0, 1, 5\right]$ & $\left(1.81,\,0.880,\,-0.449,\,4.08,\,10.9,\,6.87\right)$ \\
$\left(0, 1, 1, 0, 1, 0\right)$ & $7$ & $\left[4\right]$ & $\left(1.81,\,0.880,\,-0.449,\,4.08,\,10.9,\,6.87\right)$ \\
$\left(0, 1, 0, 1, 0, 1\right)$ & $8$ & $\left[3, 5\right]$ & $\left(1.81,\,0.880,\,-0.449,\,4.08,\,10.9,\,6.87\right)$ \\
$\left(1, 0, 1, 0, 1, 0\right)$ & $9$ & $\left[0, 4\right]$ & $\left(1.81,\,0.880,\,-0.449,\,4.08,\,10.9,\,6.87\right)$ \\
$\left(1, 0, 0, 1, 0, 1\right)$ & $10$ & $\left[0, 3, 5\right]$ & $\left(1.81,\,0.880,\,-0.449,\,4.08,\,10.9,\,6.87\right)$ \\
$\left(1, 1, 0, 0, 1, 0\right)$ & $11$ & $\left[0, 1, 4\right]$ & $\left(1.81,\,0.880,\,-0.449,\,4.08,\,10.9,\,6.87\right)$ \\
$\left(0, 1, 0, 1, 1, 0\right)$ & $12$ & $\left[3, 4\right]$ & $\left(1.81,\,0.880,\,-0.449,\,4.08,\,10.9,\,6.87\right)$ \\
$\left(1, 0, 0, 1, 1, 0\right)$ & $13$ & $\left[0, 3, 4\right]$ & $\left(1.81,\,0.880,\,-0.449,\,4.08,\,10.9,\,6.87\right)$ \\
\end{tabular}
    \caption{Broken line shelling for Catalan matroid}
    \label{fig:table_catalan_good_poset}
\end{figure}

\begin{figure}[H]
    \centering
    \includegraphics[width=0.35\textwidth]{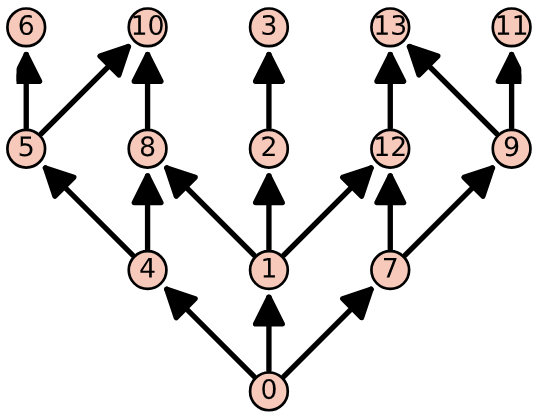}
    \caption{Poset for Catalan broken line shelling, divisibility poset}
    \label{fig:poset_catalan_good}
\end{figure}

We conclude this section with two collections of posets obtained from some brief experimentation with broken line shellings on the Catalan matroid and the graphic matroid on $K_4$, displayed in Figures \ref{fig:catalan-poset-zoo} and \ref{fig:graphic-poset-zoo}. We invite the reader to check, for example, which posets satisfy the greedoid property of Dawson or the lattice property of Las Vergnas, i.e. an extension of Theorem \ref{thm:Dawson-Chari}. Notice that posets $1,2,5$ and $6$ in the $K_4$ example immediately violate the purity and fail both conditions.

\begin{figure}[!htb]
    \centering
    \includegraphics[width=0.2\textwidth]{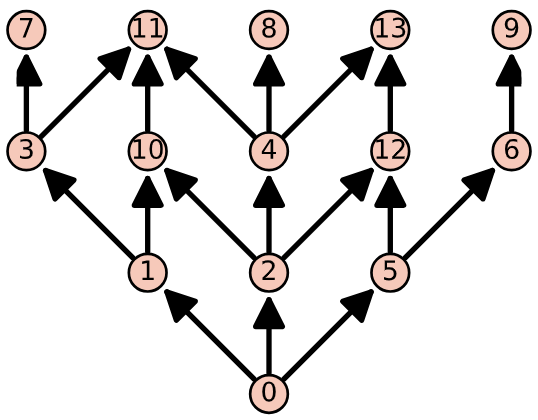} \hspace{1cm} \includegraphics[width=0.2\textwidth]{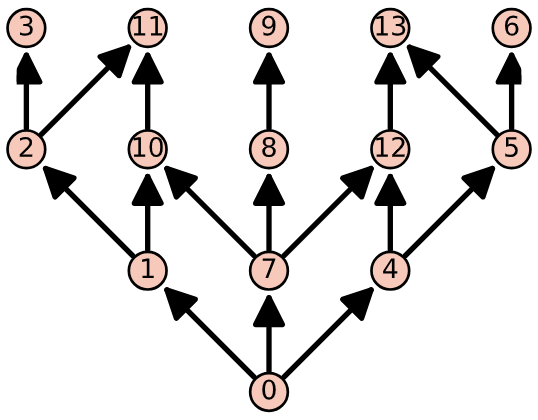} \hspace{1cm} \includegraphics[width=0.2\textwidth]{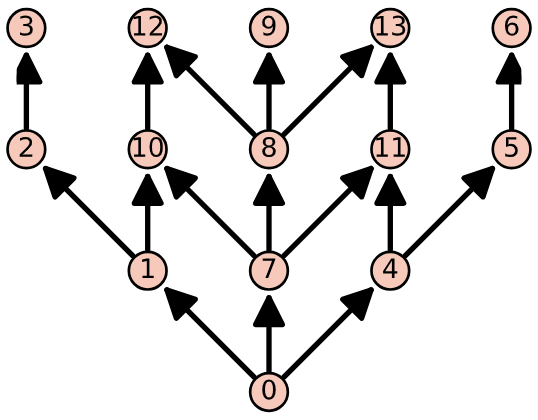}\\ \includegraphics[width=0.2\textwidth]{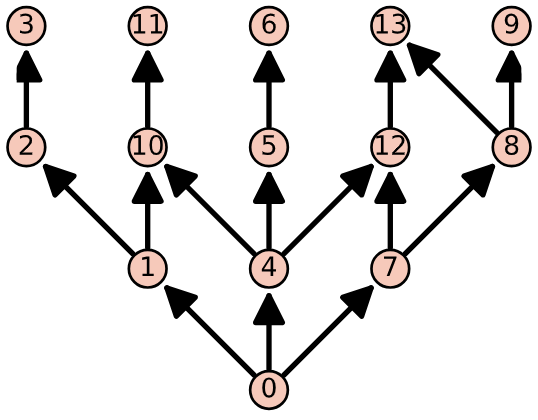} \hspace{1cm} \includegraphics[width=0.2\textwidth]{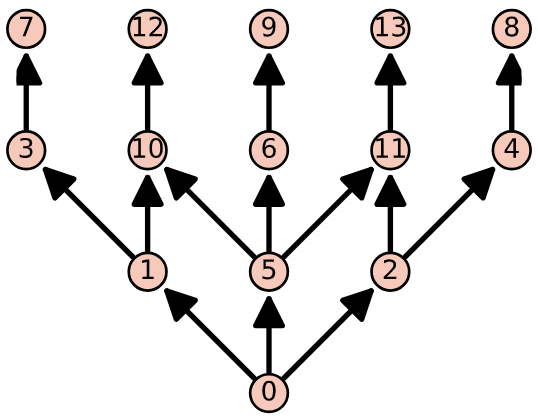} \hspace{1cm} \includegraphics[width=0.2\textwidth]{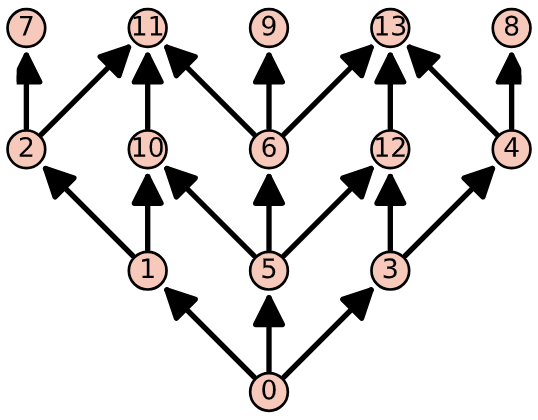}
    \caption{Non-isomorphic animals from the Catalan matroid poset zoo}
    \label{fig:catalan-poset-zoo}
\end{figure}

\begin{figure}[!htb]
    \centering
    \includegraphics[width=0.24\textwidth]{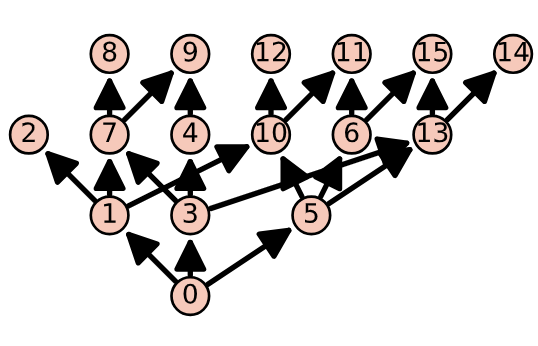} \includegraphics[width=0.24\textwidth]{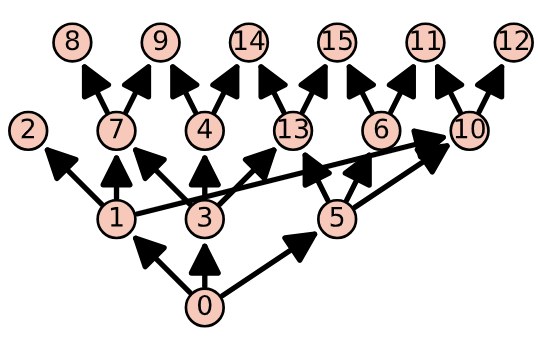} \includegraphics[width=0.24\textwidth]{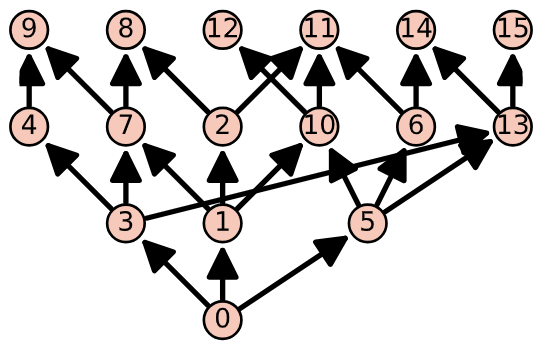}\\ \includegraphics[width=0.24\textwidth]{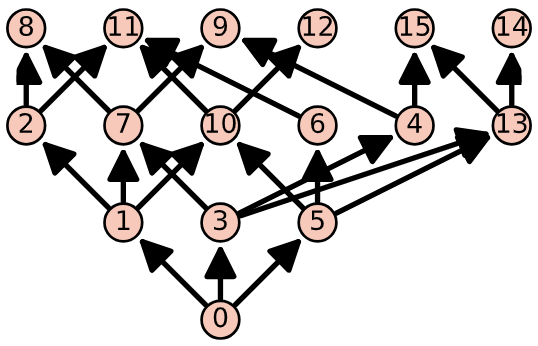} \hspace{1cm} \includegraphics[width=0.24\textwidth]{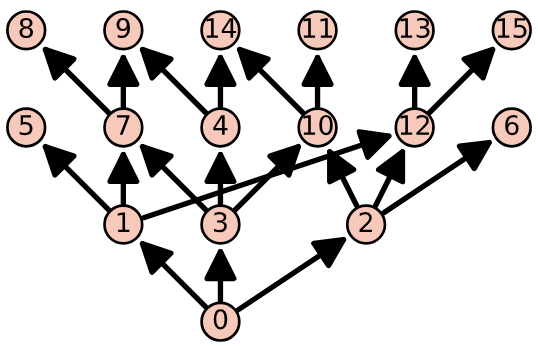}\\ \includegraphics[width=0.24\textwidth]{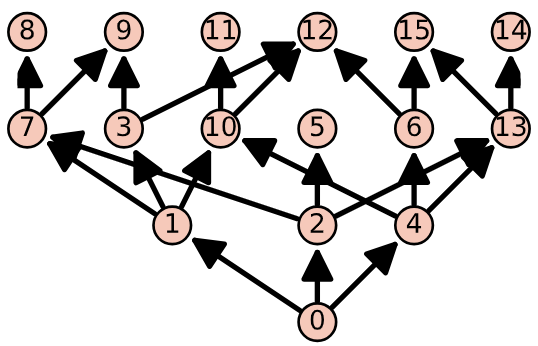} \hspace{1cm}
    \includegraphics[width=0.24\textwidth]{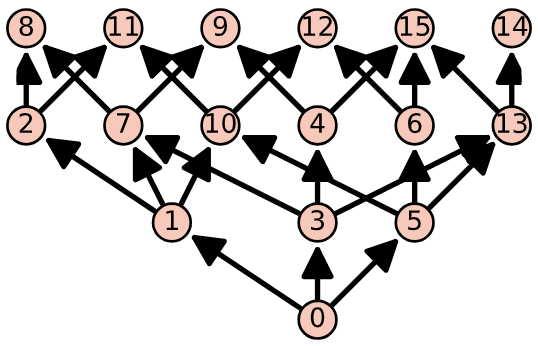}
    \caption{Some shelling order posets for the graphic matroid of the complete graph $K_4$.}
    \label{fig:graphic-poset-zoo}
\end{figure}

\section{Questions and future directions}\label{section:questions}
To finalize, we point out a few directions for potential future research (in addition to the Questions \ref{ques:classify-reversible}, \ref{ques:construct-weakly-geometric}, \ref{ques:dual-matroid-extendably-shellable}, \ref{ques:boundaries-crosspolytopes}, \ref{ques:acyclic-orientation} in the text above). We believe the new results in this article, together with the software package, open up many new directions.

First of all, we focused on (pinned) broken line shelling orders. But, there are many other shelling orders, though we cannot quantify this yet. For example, it would be interesting to investigate the difference between geometric and combinatorial shelling orders. Since the most important feature seems to be the restriction set poset, consider the following. For a matroid $M$ let $sg(M)$ be the number of distinct restriction set posets for geometric shelling orders, and let $sc(M)$ be the number of distinct restriction set posets for any shelling order. This ratio describes what the geometry can tell us about the combinatorics. First, the difference between geometric and combinatorial properties is interesting in many situations, and this instance is no exception. Second, it may help us to assess the difference between Simon's conjecture and Conjecture \ref{conj:hypersimplex-extendably-shellable}.

\begin{ques}
Are there interesting sequences of matroids $\{M_n\}_{n=1}^\infty$ for which we can understand the fraction
\begin{equation*}
    \frac{sg(M_n)}{sc(M_n)}, \text{ as } n \to \infty?
\end{equation*}
\end{ques}

Suggestions would include uniform matroids or Catalan matroids. These both have other growth patterns which are interesting.

Since the internally passive posets of line shellings are highly structured, we would like to find conditions on broken line shellings so that their restriction set posets share these features. For example, what conditions on broken line shellings might guarantee the restriction set poset has a greedoid structure? The same question applies for the lattice property of Las Vergnas (see Figure \ref{fig:graphic-poset-zoo} for several broken line shelling restriction set posets violating these). Inspired by the various ``flip'' techniques from topological combinatorics, we propose to identify the effects of removing or adding a single pivot. One would hope for a simple set of rules which governs how the restriction set posets change due to certain flips. We hope the software will be instrumental in addressing these questions.

\begin{ques}
Assume there is a broken line shelling whose restriction set poset is graded. What conditions guarantee the graded-ness is preserved upon removal/addition of a single pivot?
\end{ques}

\begin{ques}
Assume there is a broken line shelling whose restriction set poset is a greedoid. What conditions guarantee the greedoid property is preserved upon removal/addition of a single pivot?
\end{ques}

\begin{ques}
Assume there is a broken line shelling whose restriction set poset becomes a lattice after adding an artificial maximum. What conditions guarantee this lattice property is preserved upon removal/addition of a single pivot?
\end{ques}

We believe that these are key steps in understanding Stanley's conjecture. To elaborate, the restriction set posets coming from line shellings are quite similar to face posets of pure multicomplexes. Since the divisibility conditions can be checked locally on a poset, one would hope that there exist rules like those suggested in the questions above, which would instruct the alteration of a line shelling in order to produce a poset satisfying Stanley's conjecture. In particular, we notice that Dall established some precise combinatorial conditions for such posets to be multicomplexes in the case of line shellings. It is reasonable to expect that such conditions can be rephrased in terms of the shelling order instead of the underlying ground set order. We hope the geometry behind the algorithm for constructing broken line shelling posets may provide a key source of insight into what these conditions might look like.

Consider the refinement of Stanley's conjecture presented in \cite{KleeSam2015}. Besides Dall's construction of internally perfect matroids, there are other examples for which this conjecture holds. For instance, the class of Schubert matroids and also the class of positroids were shown to satisfy the conjecture in \cite{Samper} and \cite{OhThesis}, respectively. One might hope to reprove these results using the geometric techniques presented here, and in the process learn further insights that could help to extend the results to other classes of matroids.

Next, recall that line shellings of general polytopes introduce combinatorial data for computing the cd-index of a polytope as done by Lee in \cite{LeeCD}. This is a polynomial in the non-commutative variables $c$ and $d$ that keeps track of flags of faces of the polytope in a condensed way. In light of Theorem \ref{thdm:shell}, we believe that there may be an interesting connection between the cd-index coefficients of the dual matroid polytope and the $h$-vector of the independence complex of the matroid.

\begin{ques}
Let $M$ be a matroid. Is there a relationship between the $cd$-index of the polytope $P_M^*$ and the $h$-vector of the independence complex $\I(M)$? Is it possible to recover one of the two invariants from the other one? 
\end{ques}

As mentioned before, the internal activity of a line shelling recovers an evaluation of the Tutte polynomial. To recover the entire polynomial one has to mix internal and external activity simultaneously. A natural question that arises is if the geometric methods of this article can be extended to recover the entire Tutte polynomial as a shellability invariant.

Very recent work of Ardila, Denham, and Huh \cite{ArdilaDenhamHuh} proved a variety of old conjectures for the $h$-vector of matroids and broken circuit complexes. It is tempting to study some of their results through the light of new geometric methods. As a first step, one could try to define a general version of the broken circuit complex depending on a (pinned) broken line shelling, and the geometric nature of our arguments seems to blend well with the spirit and suggestions of their work.

Finally, and perhaps more speculatively, we recall that geometric lattices can also be defined in terms of shelling orders by constructing EL-labellings based on permutations of the atoms of the lattice \cite{Hersh}. A geometric lattice is purely combinatorial, but its geometric analog is the Bergman fan \cite{ArdilaKlivans2006BergmanComplexMatroidPhylogeneticTrees}. One may wonder if, by using this connection, we might find a large class of shellings parametrized using the polyhedral structure of this fan.

\noindent{\bf Acknowledgements }
Both authors would like to especially thank the entire staff at MPI-MIS Leipzig for the great support they have provided during the complicated times of the crisis. Our ability to complete this work from home in the middle of the pandemic is in big part due to their effort to keep us well connected and with full access to the needed resources. We appreciate their effort. This project was born as the second author was a postdoctoral fellow at the University of Miami. He wants the thank the institution for providing and excellent working environment. He is also thankful to Richard Stanley for some interesting suggestions at the beginning of the project.

\bibliographystyle{plain}
\bibliography{references}

\end{document}